\documentclass[11pt,letterpaper,twoside,reqno]{amsart}
\usepackage[margin=1in]{geometry}

\usepackage[utf8]{inputenc}
\usepackage{fancyhdr}
\usepackage[foot]{amsaddr}
\usepackage{amsmath}
\usepackage{amsfonts}
\usepackage{amssymb}
\usepackage{graphicx}
\usepackage{mathtools}
\usepackage{dsfont}
\usepackage[mathscr]{euscript}
\usepackage{mathrsfs}
\DeclareSymbolFont{rsfs}{U}{rsfs}{m}{n}
\DeclareSymbolFontAlphabet{\mathscrsfs}{rsfs}
\usepackage{tikz-cd}
\usepackage{thm-restate}

\usepackage{orcidlink}
\usepackage{xcolor}

\definecolor{CarmineRed}{rgb}{0.585938, 0, 0.09375}
\definecolor{UltraMarine}{rgb}{0.0703125, 0.0390625, 0.558594}

\usepackage{hyperref}
\hypersetup{colorlinks, linkcolor={CarmineRed}, citecolor={UltraMarine},urlcolor={UltraMarine} }
\usepackage{cleveref}

\makeatletter
\def\subsubsection{\@startsection{subsubsection}{3}%
\z@{.5\linespacing\@plus.7\linespacing}{-.5em}%
{\normalfont\itshape}}
\makeatother

\newcommand{\D}{{\mathop{}\!\mathrm{d}}} 
\newcommand{\R}{\mathbb{R}}
\newcommand{\N}{\mathbb{N}}
\newcommand{\PP}{\mathbb{P}}
\newcommand{\E}{\mathbb{E}}
\newcommand{\A}{\mathscr{A}}
\newcommand{\K}{{\operatorname{K}}}
\newcommand{\bS}{\mathscr{S}}
\newcommand{\X}{\mathscr{X}}
\newcommand{\Y}{\mathscr{Y}}
\newcommand{\cD}{\mathscrsfs{D}}
\newcommand{\KL}{D_{\textup{KL}}}

\DeclareMathOperator*{\argmax}{arg\,max}
\DeclareMathOperator*{\argmin}{arg\,min}
\DeclareRobustCommand{\leftmapsto}{\text{\reflectbox{$\mapsto$}}}

\numberwithin{equation}{section}  

\newtheorem{definition}{Definition}[section]

\newtheorem{remark}[definition]{Remark}
\newtheorem{theorem}[definition]{Theorem}
\newtheorem{proposition}[definition]{Proposition}
\newtheorem{corollary}[definition]{Corollary}
\newtheorem{lemma}[definition]{Lemma}

\newtheorem{setting}[definition]{Setting}
\newtheorem{asu}[definition]{Assumption}

\title[Optimal Rates of Convergence for Entropy Regularization in MDPs]{
    Optimal Rates of Convergence for
Entropy Regularization in Discounted Markov Decision Processes
} 

\author{Johannes M\"{u}ller\textsuperscript{1,$\heartsuit$,\orcidlink{0000-0001-8729-0466}}}
\email{johannes.christoph.mueller@tu-berlin.de}
\author{Semih Cayci\textsuperscript{2,\orcidlink{0000-0001-7928-0794}}}
\email{cayci@mathc.rwth-aachen.de}

\address{$^2$ Institut für Mathematik, Technische Universität Berlin, Berlin, 10623, Germany}
\address{$^2$ Department of Mathematics, RWTH Aachen University, Aachen, 52062, Germany}
\address{$^\heartsuit\!$ Corresponding author}

\begin{document}

\begin{abstract}
We study the error introduced by entropy regularization in infinite-horizon discrete discounted Markov decision processes. 
We show that this error decreases exponentially in the inverse regularization strength, both in a weighted KL-divergence and in value with a problem-specific exponent. 
This is in contrast to previously known estimates, of the order $O(\tau)$, where $\tau$ is the regularization strength.
We provide a lower bound that matches our upper bound up to a polynomial term, thereby characterizing the exponential convergence rate for entropy regularization. 
Our proof relies on the observation that the solutions of entropy-regularized Markov decision processes solve a gradient flow of the unregularized reward with respect to a Riemannian metric common in natural policy gradient methods. 
This correspondence allows us to identify the limit of this gradient flow as the generalized maximum entropy optimal policy, thereby characterizing the implicit bias of this gradient flow, which corresponds to a time-continuous version of the natural policy gradient method. 
We use our improved error estimates to show that for entropy-regularized natural policy gradient methods, the overall error decays exponentially in the square root of the number of iterations, improving over existing sublinear guarantees. 
Finally, we extend our analysis to settings beyond the entropy. 
In particular, we characterize the implicit bias regarding general convex potentials and their resulting generalized natural policy gradients. 
\end{abstract}

\maketitle

\section{Introduction}
Entropy regularization plays an important role in reinforcement learning and is usually employed to encourage exploration, thereby improving sample complexity and improving convergence of policy optimization techniques~\cite{williams1991function,sutton2018reinforcement, szepesvari2022algorithms, peters2010relative,mnih2016asynchronous}, where we refer to~\cite{neu2017unified} for an overview of different approaches.
One benefit of entropy regularization is that it leads to a tamer loss landscape~\cite{ahmed2019understanding}, and indeed it corresponds to a strictly convex regularizer in the space of state-action distribution, where the reward optimization problem is equivalent to a linear program, see~\cite{neu2017unified}. 
Employing this hidden strong convexity, vanilla and natural policy gradient methods with entropy regularization have been shown to converge exponentially fast for discrete and continuous problems, for gradient ascent and gradient flows, and for tabular methods as well as under function approximation~\cite{mei2020global, cen2021fast, cayci2021linear, muller2024geometry, kerimkulov2023fisher}. 
However, the improved convergence properties of entropy-regularization in policy optimization come at the expense of an error introduced by the regularization, for which a concise characterization remains elusive. 
A general theory of regularized Markov decision processes was established in~\cite{geist2019theory}, where for a bounded regularizer and regularization strength $\tau$, an $O(\tau)$ estimate on the regularization error is provided. 
This guarantee covers general regularizers but neither uses the structure of Markov decision processes nor the entropic regularizer, which is common in reinforcement learning. 
Combining this estimate with the $O(e^{-\tau\eta k})$ convergence guarantees of entropy-regularized natural policy gradient with tabular softmax policies this yields $O(\frac{\log k}{\eta k})$ convergence of the overall error, where $k$ denotes the number of iterations and $\tau = \frac{\log k}{\eta k}$, see~\cite{cen2021fast, sethi2024entropy}. 
We exploit the structure of the reward and entropy function and provide an explicit analysis of the entropy regularization error. 

\subsection{Contributions}
The main contribution of this article is to identify the optimal exponential convergence rates of the entropy regularization error in infinite-horizon discrete discounted Markov decision processes. 
More precisely, we summarize our contributions as follows: 
\begin{itemize}
    \item \emph{Pythagoras in Kakade divergence.} 
    We provide a Pythagorean theorem for $s$-rectangular policy classes and the Kakade divergence, which we consider as a regularizer and which is not a Bregman divergence, see \Cref{prop:pythagoras}. 
    \item \emph{Expression of Kakade gradient flows.} We study a Riemannian metric introduced by Kakade in the context of natural policy gradients and provide an explicit expression of the gradient flows of the reward with respect to this metric via the advantage function, see \Cref{thm:PGtheorem}.  
    \item \emph{Central path property.} 
    We show that the optimal policies of the entropy-regularized reward solve the gradient flow of the reward with respect to the Kakade metric, see \Cref{cor:implications}. 
    \item  \emph{Implicit bias.} 
    The correspondence between the Kakade gradient flow and the solutions of the entropy-regularized problems allows us to show that the gradient flows corresponding to natural policy gradients converge towards the generalized maximum entropy optimal policy, thereby characterizing the implicit bias of these methods, see \Cref{cor:implications}
    \item \emph{Optimal rates for entropy regularization.} 
    In \Cref{thm:convergenceValue} and \Cref{thm:convergencePolicies}, we give up to polynomial terms matching exponential $O(e^{-\Delta \tau^{-1}})$ upper and lower bounds on the entropy regularization error for a problem-dependent exponent $\Delta>0$. 
    \item \emph{Lower bound for unregularized tabular NPG.} We provide an anytime analysis of unregularized tabular natural policy gradients and show that the exponential convergence rate $O(e^{-\Delta \eta k})$ is tight up to polynomial terms, see \Cref{thm:convregenceNPG}. 
    \item \emph{Overall error analysis.} 
    We combine our estimate on the entropy regularization error with existing guarantees for regularized natural policy gradient methods and show that the error decreases exponentially in the square root of the number of iterations, see \Cref{thm:overallError}. 
\end{itemize}

\subsection{Related works}
A general theory of regularized Markov decision processes was established in~\cite{geist2019theory}. 
Here, $O(\tau)$ convergence was established for bounded regularizers, which was subsequently used in the overall error analysis of entropy-regularized natural policy gradients~\cite{dai2018sbeed,lee2018sparse}. 
Whereas the $O(\tau)$ estimate on the regularization error holds for general regularizers, it neither employs the specific structure of Markov decision processes nor the entropy function. 
In contrast, we show that the optimal entropy-regularized policies solve a gradient flow and use this to provide a $\tilde{O}(e^{-\Delta \tau^{-1}})$ estimate on the entropy regularization error, where $\tilde{O}$ hides a polynomial term in $\tau^{-1}$. 

At the core of our argument lies the observation that the optimal entropy-regularized policies solve the gradient flow of the unregularized reward with respect to the Kakade metric, which can be seen as the continuous time limit of unregularized natural policy gradient methods for tabular softmax policies.
This correspondence uses the isometry between the Kakade metric and the metric induced by the conditional entropy on the state-action distributions established in~\cite{muller2024geometry}. 
This allows us to use tools from the theory of Hessian gradient flows in convex optimization~\cite{alvarez2004hessian, mueller2023thesis}. 
A similar approach has been taken recently for the study of Fisher-Rao gradient flows of linear programs and state-action natural policy gradients~\cite{muller2024fisher}. 

Natural policy gradients have been shown to converge at a $O(\frac1k)$ rate~\cite{agarwal2021theory} which was used in~\cite{khodadadian2022linear} to show  assymptotic $O(e^{-c k})$ convergence for all $c<\Delta$. 
Our continuous time analysis uses a similar line of reasoning, but we provide an anytime analysis, a lower bound matching up to a polynomial term, and show convergence of the policies towards the generalized maximum entropy optimal policy. 
Further, we generalize our anytime analysis and optimal exponential rates to time natural policy gradient methods, extending the existing asymptotic upper and lower bounds given in~\cite{khodadadian2022linear, liu2024elementary}.  

A continuous-time analysis of gradient flows with respect to the Kakade metric has been conducted for Markov decision processes with discrete and continuous state and action spaces in~\cite{muller2024geometry, kerimkulov2023fisher}. 
Under the presence of entropy regularization with strength $\tau>0$, exponential $O(e^{-\tau t})$ convergence was established in~\cite{muller2024geometry, kerimkulov2023fisher}, which was generalized to sublinear convergence for certain decaying regularization strengths in~\cite{sethi2024entropy}. 
For the case of unregularized reward, exponential convergence was established in~\cite{muller2024geometry} without control over the exponent or the coefficient. 
For gradient flows with respect to the Fisher metric in state-action space, exponential convergence $O(e^{-\delta t})$ with a problem-specific exponent was established in~\cite{muller2024fisher}, however, a lower bound is missing there, and the exponent $\delta\le\Delta$ is dominated by the one for Kakade gradient flows. 
For Riemannian metrics corresponding to $\beta$-divergences, $O(t^{\beta^{-1}})$ convergence has been established for $\beta\in(-1, 0)$~\cite{muller2024geometry}. 
    
The algorithmic bias of policy optimization techniques has been studied before. Here, a $O(\log k)$ bound on the distance -- measured in the Kakade divergence -- of the policies to the maximum entropy optimal policy was established for an unregularized natural actor-critic algorithm~\cite{hu2021actor}. 
Working in continuous time allows us to show exponential convergence towards the generalized maximum entropy optimal policy. 
A similar algorithmic bias showing convergence to the maximum entropy optimal policy was established for natural policy gradient methods that decrease regularization and increase step sizes during optimization~\cite{li2023homotopic}. 
In contrast to our implicit bias result, this approach considers an asymptotically vanishing but explicit regularization. 

\subsection{Notation}
For a set $\X$ we denote the free vector space over $\X$ by $\R^\X \coloneqq \{ f\colon\X\to\R \}$ and similarly we define $\R^\X_{>0}$ and $\R^\X_{\ge0}$ as all positive and non-negative real functions on $\X$. 
For a finite set $\X$ we denote the support of $\mu\in\R^\X$ by $\operatorname{supp}(\mu) \coloneqq \{ x\in\X : \mu(x) \ne0 \}$. 
For a convex differentiable function $\phi\colon\Omega\to\R$ defined on a convex subset $\Omega\subseteq\R^\X$ we denote the corresponding \emph{Bregman divergence} by $D_\phi(\mu, \nu) \coloneqq \phi(\mu) - \phi(\nu) - \nabla\phi(\nu)^\top(\mu-\nu)$. 
We denote the \emph{Shannon entropy} of a vector $\mu\in\R_{>0}^\X$ by $H(\mu)\coloneqq - \sum_{x\in\X} \mu(x)\log \mu(x)$. The Bregman divergence induced by the negative Shannon entropy $-H$ is given by \emph{Kullback-Leibler} or \emph{KL-divergence} that we denote by $\KL(\mu, \nu)\coloneqq \sum_{x\in\X} \mu(x) \log \frac{\mu(x)}{\nu(x)}$. 
The \emph{probability simplex} $\Delta_\X\coloneqq\{ \mu\in\R_{\ge0}^\X : \sum_{x\in\X} \mu(x) =1 \}$ over a finite set $\X$ denotes the set of all probability vectors. 
Given a second finite set $\Y$, we can identify the Cartesian product $\Delta_\X^\Y = \Delta_\X\times\cdots \times\Delta_\X$ with the set of stochastic matrices or the set of Markov kernels. 
We refer to $\Delta_\X^\Y$ as the \emph{conditional probability polytope} and for $P\in\Delta_\X^\Y$ we write $P(x|y)$ for the entries of the Markov kernel. 
For $\mu\in\Delta_\Y$ and a Markov kernel $P\in\Delta_\X^\Y$ we denote the corresponding joint probability measure over $\X\times\Y$ by $\mu\ast P\in\Delta_{\X\times\Y}$ given by $(\mu\ast P)(x,y) \coloneqq \mu(x)P(y|x)$. 
Finally, for two vectors $\mu, \nu\in\R^\X$ we denote the \emph{Hadamard product} by $\mu\odot\nu\in\R^\X$ with entries $(\mu\odot\nu)(x) \coloneqq \mu(x)\nu(x)$.

\section{Preliminaries}\label{sec:preliminaries}

In this section, we provide background material from the theory of Markov decision processes. 
We put an emphasis on regularization and see that entropy-regularized Markov decision processes are equivalent to a regularized linear programming formulation of the Markov decision process, where the regularizer is given by a conditional entropy term. 
We conclude with a general discussion of regularized linear programs and revisit the central path property. 
This states that the solutions of regularized linear programs with regularization strength $t^{-1}$ solve the gradient flow of the linear program with respect to the Riemannian metric induced by the convex regularizer. 
Our explicit analysis of the entropy regularization error is built on this dynamic interpretation of the solutions of the optimizers of the regularized problems. 

\subsection{Markov decision processes and entropy regularization} 

We consider a finite set $\bS$ of states of some system and a finite set $\A$ of actions that can be used to control the state $s\in\bS$. 
The transition dynamics are described by a fixed Markov kernel $P\in\Delta_{\bS}^{\bS\times\A}$, where $P(s'|s,a)$ describes the probability of transitioning from $s$ to $s'$ under the action $a$. 
It is the goal in Markov decision processes and reinforcement learning to design a \emph{policy} $\pi\in\Delta_\A^\bS$, where the entries $\pi(a|s)$ of the stochastic matrix describe the probability of selecting an action $a$ when in state $s$. 
A common optimality criterion is the \emph{discounted infinite horizon reward} given by 
\begin{align}
    R(\pi) \coloneqq (1-\gamma) \E_{\PP^{\pi, \mu}}\left[ \sum_{t\in\N} \gamma^t r(S_t, A_t)  \right], 
\end{align}
where $\mu\in\Delta_\bS$ is an initial distribution over the states and $\PP^{\pi, \mu}$ denotes the law of the Markov process on $\bS\times\A$ induced by the iteration 
\begin{align}
    S_0\sim\mu, A_t \sim\pi(\cdot|S_t), S_{t+1}\sim P(\cdot|S_t, A_t),
\end{align}
$r\in\R^{\bS\times\A}$ is the \emph{instantaneous reward} vector, and $\gamma\in[0,1)$ is the \emph{discount factor}. 
To encourage exploration, it is common to regularize the reward with an entropy term, which yields the \emph{entropy-regularized} or \emph{KL-regularized} reward 
\begin{align}
    R_\tau^\mu(\pi) \coloneqq (1-\gamma) \E_{\PP^{\pi, \mu}}\left[ \sum_{t\in\N} \gamma^t (r(S_t, A_t) - \tau \KL(\pi(\cdot|S_t), \pi_0(\cdot|S_t)) \right] 
\end{align}
for some \emph{reference policy} $\pi_0\in\Delta_\A^\bS$ and \emph{regularization strength} $\tau\ge0$. 
The {entropy or KL-regularized reward optimization problem} is given by 
\begin{align}\label{eq:regularizedProblem}
    \max R_\tau(\pi) \quad \text{subject to } \pi\in\Delta_\A^\bS. 
\end{align} 
The regularization can be interpreted as a convex regularization and consequently was used to show exponential convergence~\cite{mei2020global, cen2021fast, cayci2021linear, lan2022policy}, however, it introduces an error and leads to a new optimal policy $\pi^\star_\tau$, which might not be optimal with respect to the unregularized reward $R$. 
It is our goal to understand how well the entropy-regularized reward optimization problem~\eqref{eq:regularizedProblem} approximates the unregularized reward optimization problem. 
For this we provide an explicit analysis of the \emph{entropy regularization error} $R^\star - R(\pi^\star_\tau)$ and $\min_{\pi^\star\in\Pi^\star}D(\pi^\star, \pi_\tau^\star)$, where $R^\star = \max_{\pi} R(\pi)$ denotes the optimal reward, $\Pi^\star$ the set of optimal policies, and $D(\cdot, \cdot)$ is some notion of distance. 

A central role in theoretical and algorithmic approaches to Markov decision processes and reinforcement learning plays the \emph{value function} $V^\pi\in\R^\bS$ given by 
\begin{align}
    V^\pi_{\tau}(s) \coloneqq R^{\delta_s}_{\tau}(\pi) \quad \text{for all } s\in\bS,
\end{align}
which stores the reward obtained when starting in a deterministic state $s\in\bS$. 
The unregularized and regularized \emph{state-action} or $Q$-\emph{value functions} are given by
\begin{align}
    Q_{\tau}^\pi(s,a) \coloneqq (1-\gamma) r(s,a) + \gamma \sum_{s'\in\bS} V^{\pi}_{\tau}(s') P(s'|s,a) \quad \text{for all } s\in\bS, a\in\A. 
\end{align}
The \emph{advantage function} of a policy $\pi$ given by 
\begin{align}
    A_{\tau}^\pi(s,a) \coloneqq Q^\pi_{\tau}(s,a) - V^\pi_{\tau}(s),
\end{align}
in words, $A_{\tau}^\pi(s,a)$ describes how much better is it to select action $a$ and follow the policy $\pi$ afterwards compared to following $\pi$. 
Finally, we define the \emph{optimal reward} and \emph{optimal value functions} via $R^\star\coloneqq\max_{\pi\in\Delta_\A^\bS}R(\pi)$ and 
\begin{align}
    V^{\star}_{\tau}(s) \coloneqq \max_{\pi\in\Delta_\A^\bS} V_{\tau}^{\pi}(s) \quad \text{and } Q_{\tau}^{\star}(s,a) \coloneqq \max_{\pi\in\Delta_\A^\bS} Q_{\tau}^{\pi}(s,a) \quad \text{for } s\in\bS, a\in\A
\end{align}
and the \emph{optimal advantage function} via 
\begin{align}
    A_{\tau}^{\star}(s,a) \coloneqq Q_{\tau}^{\star}(s,a) - V^\star_\tau(s) \quad \text{for } s\in\bS, a\in\A. 
\end{align}
It is well known that there are optimal policies $\pi^\star_{\tau}\in\Delta_\A^\bS$ satisfying $V^\star_{\tau} = V^{\pi^\star_{\tau}}$ and $Q^\star_{\tau} = Q^{\pi^\star_{\tau}}$, which is known as the Bellman principle, see~\cite{geist2019theory}. 
In the unregularized case, the optimal advantage function $A^\star$ satisfies $A^\star(s,a)\le0$ for all $s\in\bS, a\in\A$, and an action $a\in\A$ is optimal in a state $s\in\bS$ if and only if $A^\star(s,a)=0$.
For $s\in\bS$, we denote the set of optimal actions by $A^\star_s\coloneqq\{a\in\A: A^\star(s,a)=0\}$ and the set of optimal policies is given by 
\begin{align}
    \Pi^\star 
    \coloneqq \left\{ \pi\in\Delta_\A^\bS : R(\pi) = R^\star \right\} 
    = \left\{ \pi\in\Delta_\A^\bS : \operatorname{supp}(\pi(\cdot|s)) \subseteq A^\star_s \text{ for all } s\in\bS \right\}. 
\end{align}
We denote the unregularized value and advantage functions by $V^\pi=V^\pi_0$, $Q^\pi=Q^\pi_0$, $q^\pi=q^\pi_0$, $A^\pi=A^\pi_0$, and $B^\pi=B^\pi_0$. Note that $Q^\pi = q^\pi$ and $A^\pi = B^\pi$. 

Of central importance in the theory of Markov decision processes are the \emph{state distributions}, which measure how much time the process spends at the individual states for a given policy $\pi\in\Delta_\A^\bS$. This is formalized by 
\begin{align}
    d^{\pi}(s) = d^{\pi, \mu}_\gamma(s) \coloneqq (1-\gamma) \sum_{t\in\N} \gamma^t \PP^{\pi, \mu}(S_t=s). 
\end{align}
Note that indeed $d^\pi\in\Delta_\bS$ by the geometric series. 
Sometimes, the state distributions are called state frequencies, state occupancy measures, or (state) visitation distributions. 
We work with the following notion of distance between policies. 

\begin{definition}[Kakade divergence]
    For two policies $\pi_1, \pi_2\in\Delta_\A^\bS$ we call 
        \begin{align}
            D_\K(\pi_1, \pi_2) = D_\K^\mu(\pi_1, \pi_2) \coloneqq \sum_{s\in\bS} d^{\pi_1}(s) \KL(\pi_1(\cdot|s),\pi_2(\cdot|s)) 
        \end{align}
    the \emph{Kakade divergence} between $\pi_1$ and $\pi_2$. 
    Note that $D_\K$ depends on the initial distribution $\mu\in\Delta_\bS$ and the discount factor $\gamma\in[0,1)$. 
\end{definition}

The Kakade divergence arises naturally when studying entropy regularization since 
\begin{align}
    R_\tau(\pi) = R(\pi) - \tau D_\K(\pi, \pi_0) \quad \text{for all } \pi\in\Delta_\A^\bS. 
\end{align}
Note that although $\KL$ is a Bregman divergence, the Kakade divergence is not, which hinders the direct use of mirror descent tools developed in the context of convex optimization~\cite{bubeck2015convex}. 

\begin{remark}[Kakade divergence is not Bregman]
Bregman divergences are well-studied in convex optimization~\cite{alvarez2004hessian, bubeck2015convex}, however, the Kakade divergence does not fall into this class. 
Indeed, Bregman divergences are convex in their first argument, which is not generally true for the Kakade divergence.  
To construct a specific example, where $D_\K(\cdot,\pi)$ is not convex, we consider the Markov decision process with two states and actions shown in \Cref{fig:kakadeExample}.  
\begin{figure}[h!]
    \centering
    \begin{tikzpicture}
    \node[shape=circle, draw=black, minimum size=1cm] (A) at (0,0) {$s_1$};
    \node[shape=circle,draw=black,minimum size=1cm] (B) at (3,0) {$s_2$};
    \path [->] (B) [bend left=20] edge node[below] {$a_1$} (A);
    \path [->] (A) [bend left=20] edge node[above] {$a_2$} (B);
    \path [->] (A) [loop left] edge node[left] {$a_1$} (A);
    \path [->] (B) [loop right] edge node[right] {$a_2$} (B);
\end{tikzpicture}
\caption{Transition graph of the Markov decision process.} 
\label{fig:kakadeExample}
\end{figure}
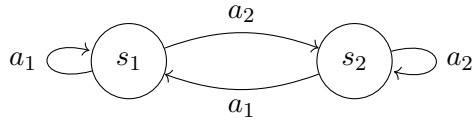
Further, we choose the reference policy $\pi$ to be the uniform policy and if we consider $\pi_p(a_1|s_i) = p$ for $i=1,2$, then we obtain 
\begin{align}
    g(p)\coloneqq D_\K(\pi_p, \pi) = (1-\gamma)\phi(p) + \gamma^2 p^2 \phi(p),
\end{align}
where $\phi(p) \coloneqq -p\log p - (1-p)\log(1-p)$. 
Taking the second derivative yields 
\begin{align}
    \partial_p^2g(p) = -(1-\gamma)\left(\frac1{1 - p} - \frac1{p^2}\right) - \gamma^2 \left(\frac1{1-p} + \frac1p\right) p - 2 \gamma^2 (\log p-\log(1 - p)).
\end{align}
Taking $p\to1$ yields $\partial_p^2g(p)\to-\infty$ showing the non-convexity of $g$ and therefore $D_\K(\cdot, \pi)$. 
\end{remark}

As we discuss later, the Kakade divergence is the pullback of the conditional KL divergence on the space of state-action distributions, which renders the entropy-regularized reward optimization problem equivalent to a linear program with a Bregman regularizer, see \Cref{prop:SAGeometryLP}. 

The following sublinear estimate on the regularization error is well known and commonly applied when approximating unregularized Markov decision processes via regularization, see~\cite{geist2019theory, cen2021fast}. 
Here, we follow the terminology common in optimization theory, where linear convergence refers to exponential convergence $O(e^{-ct})$ and sublinear to an algebraic convergence rate $O(t^{-\kappa})$~\cite{nesterov2018lectures}. 

\begin{proposition}[Sublinear estimate on the regularization error]\label{prop:generalEstimate}
    Let $\pi^\star_\tau$ denote an optimal policy of the regularized reward $R_\tau(\pi)\coloneqq R(\pi) - \tau D_\K(\pi, \pi_0)$ for some $\pi_0\in\Delta_\A^\bS$ and denote the set of optimal policies by $\Pi^\star$. 
    Then we have 
    \begin{align} 
    0\le R^\star - R(\pi^\star_\tau) \le \tau\inf_{\pi^\star\in\Pi^\star} D_\K(\pi^\star, \pi_0), 
    \end{align}
    where $\inf_{\pi\in\Pi^\star}D_\K(\pi, \pi_0)<+\infty$ for $\pi_0\in\operatorname{int}(\Delta_\A^\bS)$. 
\end{proposition}
\begin{proof}
    For any $\pi^\star\in\Pi^\star$ we have 
    \begin{align*}
        R^\star-\tau D_\K(\pi^\star, \pi_0)  = R_\tau(\pi^\star)\le R_\tau(\pi^\star_\tau) \le R(\pi^\star_\tau).
    \end{align*}
    Rearranging and taking the infimum over $\Pi^\star$ yields the claim. 
\end{proof}

Note that \Cref{prop:generalEstimate} holds for general regularizers and uses neither the reward nor the regularizer. 
In the remainder, we will improve this sublinear estimate $O(\tau)$ on the suboptimality of the entropy optimal regularized policy $\pi^\star_\tau$ to an exponential estimate $\tilde O(e^{-\Delta\tau^{-1}})$, establish a lower bound matching the upper bound up a polynomial term, and provide a similar estimate on $\min_{\pi^\star\in\Pi^\star} D_\K(\pi^\star, \pi_\tau^\star)$.
This characterizes the optimal exponential convergence rate of the entropy regularization error in discounted Markov decision processes.

\subsection{State-action geometry of entropy regularization}
Similarly to the state-distributions, we define the \emph{state-action distribution} of a policy $\pi\in\Delta_\A^\bS$ via 
\begin{align}
    \nu^\pi(s,a) = \nu^{\pi, \mu}_\gamma(s,a) \coloneqq (1-\gamma) \sum_{t\in\N} \gamma^t \PP^{\pi, \mu}(S_t=s, A_t = a).
\end{align}
Note that by the geometric series and stationarity of the policy, we have  $\nu^\pi\in\Delta_{\bS\times\A}$ and it holds that $\nu^\pi(s,a) = d^\pi(s)\pi(a|s)$.  
The state and state-action distributions are also known as \emph{occupancy measures} or \emph{state frequencies}. 
An important property of state-action distributions is given by 
\begin{align}
    R(\pi) = r^\top \nu^\pi,
\end{align}
which can be seen using the dominated convergence theorem~\cite{mueller2023thesis}. 
Moreover, we have the classic characterization of the set of state-action distributions.

\begin{proposition}[State-action polytope, \cite{derman1970finite}]
    The set $\cD = \{\nu^\pi : \pi\in\Delta_\A^\bS\}$ of state-action distributions is a polytope given by 
    \begin{align}
        \cD = \Delta_{\bS\times\A} \cap \left\{ \nu\in\R^{\bS\times\A} : \ell_s(\nu)=0 \text{ for all } s\in\bS \right\},
    \end{align}
    where 
    \begin{align}
        \ell_s(\nu) \coloneqq \sum_{a\in\A}\nu(s,a) - \gamma \sum_{a'\in\A, s'\in\bS} P(s|s',a') \nu(s',a') - (1-\gamma) \mu(s). 
    \end{align} 
\end{proposition}

In particular, the characterization of the state-action distributions of a Markov decision process as a polytope shows that  the reward optimization problem is equivalent to the linear program 
\begin{align}\label{eq:MDP-LP}
    \max r^\top \nu \quad \text{subject to } \nu\in\cD. 
\end{align}
Further, it is well known that for a state-action distribution $\nu\in\cD$, a policy $\pi\in\Delta_\A^\bS$ with $\nu^\pi = \nu$ can be computed by conditioning, see for example~\cite{kallenberg1994survey, muller2022pomdps, laroche2023occupancy}, and hence we have  
\begin{equation}\label{eq:conditioning}
    \pi(a|s) = \nu(a|s) \coloneqq \begin{cases}
        \frac{\nu(s,a)}{\sum_{a'} \nu(s,a')} \quad & \text{if }\sum_{a'} \nu(s,a')>0 \text{ and} \\[\medskipamount] 
        \frac1{\lvert \A\rvert} \quad & \text{otherwise}. 
    \end{cases}
\end{equation}
The entropy-regularized reward optimization problem admits an interpretation as a regularized version of the linear program~\eqref{eq:MDP-LP}, where the regularizer is given by the conditional entropy. 

\begin{definition}[Conditional entropy and KL]
For $\nu\in\R_{>0}^{\bS\times\A}$ we define the \emph{conditional entropy} via 
\begin{align}
    H_{A|S}(\nu) \coloneqq \sum_{s\in\bS,a\in\A} \nu(s,a) \log \frac{\nu(s,a)}{\sum_{a'} \nu(s,a')} = H(\nu) - H(d),
\end{align}    
where $d(s)\coloneqq\sum_{a\in\A}\nu(s,a)$. 
For $\nu_1, \nu_2\in\R_{>0}^{\bS\times\A}$ we call 
    \begin{align}
        D_{A|S}(\nu_1, \nu_2) \coloneqq \sum_{s\in\bS,a\in\A} \nu_1(s,a) \log\frac{\nu_1(a|s)}{\nu_2(a|s)} = \KL(\nu_1, \nu_2) - \KL(d_1, d_2). 
    \end{align}
the \emph{conditional KL-divergence} between $\nu_1$ and $\nu_2$, where $d_i(s)= \sum_{a\in\A} \nu_i(s,a)$. 
\end{definition}

Direct computation shows that $D_{A|S}$ is the Bregman divergence induced by $H_{A|S}$, see~\cite[A.1]{neu2017unified}.

\begin{proposition}[State-action geometry of entropic regularization]\label{prop:SAGeometryLP}
It holds that 
\begin{align}
    D_\K(\pi_1, \pi_2) = D_{A|S}(\nu^{\pi_1}, \nu^{\pi_2}) \quad \text{for all } \pi_1, \pi_2\in\Delta_\A^\bS 
\end{align} 
showing that $D_\K$ is the pull back of the conditional KL-divergence $D_{A|S}$.
In particular, the entropy-regularized reward optimization problem~\eqref{eq:regularizedProblem} is equivalent to the regularized linear program 
\begin{align}
    \max r^\top \nu - \tau D_{A|S}(\nu, {\nu_0})
    \quad \text{subject to } \nu\in\cD,
\end{align}
where $\nu_0 = \nu^{{\pi_0}}$, meaning that there is a unique solution $\nu^\star_\tau\in\operatorname{int}(\cD)$ and therefore a unique solution $\pi^\star_\tau\in\operatorname{int}(\Delta_\A^\bS)$ of the regularized problem and we have $\nu^{\pi^\star_\tau} = \nu^\star_\tau$. 
\end{proposition}
\begin{proof}
This is a direct consequence of $R_\tau(\pi) = r^\top \nu^{\pi} - \tau D_{A|S}(\nu, {\nu_0})$, see also~\cite{neu2017unified}. 
\end{proof}

The Pythagorean theorem can be generalized to Bregman divergences, which is well known in the field of information geometry and convex optimization~\cite{Csiszár2012,bubeck2015convex,amari2016information, ay2017information}. 
As the Kakade divergence is not a Bregman divergence, it is not included in those general results, but using the characterization of $D_\K$ as the pullback of a conditional KL-divergence $D_{A|S}$ allows us to provide a Pythagorean theorem for the Kakade divergence and $s$-rectangular policy classes. 

\begin{theorem}[Pythagoras for Kakade divergence]\label{prop:pythagoras} 
    Consider a set of policies $\Pi = \otimes_{s\in\bS} \Pi_s\subseteq\Delta_\A^\bS$ given by the cartesian product of polytopes $\Pi_s\subseteq\Delta_\A$. 
    Further, fix $\pi_0\in\Delta_\A^\bS$ and consider the Kakade
    projection 
    \begin{align}
        \hat\pi = \argmin_{\pi\in\Pi} D_\K(\pi, \pi_0)
    \end{align}
    of $\pi_0$ onto $\Pi$. 
    Then for any $\pi\in\Pi$ we have 
    \begin{align}\label{eq:pythagoras}
        D_\K(\pi, \pi_0) \ge D_\K(\pi, \hat\pi) + D_\K(\hat\pi, \pi_0). 
    \end{align}
    If further $\Pi_s = \Delta_\A\cap\mathcal L_s$ for affine spaces $\mathcal L_s\subseteq\R^\A$ for all $s\in\bS$, then we have equality in~\eqref{eq:pythagoras}. 
\end{theorem}
\begin{proof}
Consider the set $D\coloneqq\{\nu^\pi:\pi\in\Pi\}\subseteq\cD$ of state-action distributions arising from the policy class $\Pi$, which is again a polytope~\cite[Remark 55]{muller2022pomdps}. 
In particular, $D$ is convex and we can pass to the corresponding state-action distributions $\nu_0, \hat \nu, \nu$ and apply the Pythagorean theorem for Bregman divergences, see \Cref{sec:app:pythagoras}. 
If further $\Pi_s = \Delta_\A\cap\mathcal L_s$, then we have $D = \cD\cap\mathcal L$ for an affine subspace $\mathcal L$~\cite[Proposition 14]{muller2022pomdps}. 
Consequently, the Bregman projection $\hat \nu$ will always lie at the relative interior $\operatorname{int}(D)$ and we get equality in the Pythagorean theorem. 
\end{proof}

\subsection{Regularized linear programs and Hessian gradient flows}\label{subsec:H-GF}

We have seen that the state-action distributions of the optimal regularized policies $\pi^\star_\tau$ solve the linear program associated with the Markov decision process with a conditional entropic regularization. 
Here, we see that the solutions of regularized linear programs solve the gradient flow of the unregularized linear objective with respect to the Riemannian metric induced by the convex regularizer. 
We refer to~\cite{alvarez2004hessian} for a more general discussion of Hessian gradient flows, where slightly stronger assumptions on $\phi$ are made. 
We work in the following setting. 

\begin{setting}\label{setting:generalLP}
We consider the linear program 
\begin{equation}\label{eq:LP}
    \max c^\top x \quad \text{subject to } x\in P, 
\end{equation}
with cost $c\in\R^{d}$ and feasible region given by a polytope $P \subseteq \R^d$. 
Further, we consider a twice continuously differentiable convex function $\phi$ defined on a neighborhood of $\operatorname{int}(P)$ and assume that $\nabla^2\phi(x)$ is strictly positive definite on $TP$ for all $x\in\operatorname{int}(P)$, where $TP$ denotes the tangent space of the polytope $P$, which is given by the affine span. 
We define the Riemannian Hessian metric $g^\phi_x(v,w) \coloneqq v^\top \nabla^2\phi(x) w$ on $\operatorname{int}(P)$ 
and denote the gradient of $f$ with respect to $g^\phi$ by $\nabla^\phi f$. 
By $(x_t)_{t\in[0, T)}\subseteq\operatorname{int}(P)$ we denote a solution of the Hessian gradient flow
\begin{equation}\label{eq:H-GF}
    \partial_t x_t = \nabla^\phi f(x_t)
\end{equation}
with initial condition $x_0\in \operatorname{int}(P)$ and potential $f(x) = c^\top x$, where $T\in\R_{\ge0}\cup\{+\infty\}$. 
Finally, we denote the Bregman divergence induced by $\phi$ by $D_\phi$. 
\end{setting}

Note that $(x_t)_{t\in [0, T)} \subseteq \operatorname{int}(P)$ solves the Hessian gradient flow~\eqref{eq:H-GF}
if and only if we have 
\begin{equation}\label{eq:explicitFRGF}
    g_{x_t}^\phi(\partial_t x_t, v) = \langle \nabla^2 \phi(x_t) \partial_t x_t, v \rangle = \langle \nabla f(x_t), v \rangle 
    \quad \text{for all } v\in TP, t\in[0, T). 
\end{equation}
We can now formulate and prove the equivalence property between Hessian gradient flows of linear programs and solutions of Bregman regularized linear programs.

\begin{proposition}[Central path property]\label{prop:centralPathLP}
Consider Setting~\ref{setting:generalLP}.  
Then the Hessian gradient flow $(x_t)_{t\in[0,T)}$ of the linear program is characterized by 
\begin{equation}\label{eq:centralPath}
    x_t \in \argmax\left\{ c^\top x - t^{-1} D_\phi (x, x_0) : x \in P \right\} \quad \text{for all } t\in(0,T). 
\end{equation}
\end{proposition}
\begin{proof}
Note that by the first-order stationarity conditions for equality-constrained optimization, a point $\hat x\in\operatorname{int}(P)$ maximizes $g( x)\coloneqq c^\top x - t^{-1} D_\phi ( x,  x_0)$ over the feasible region $P$ of the linear program if and only if $\langle \nabla g(\hat{ x}_t), v\rangle = 0$ for all $v\in TP$. 
Direct computation yields $\nabla g({ x}) = c - t^{-1} (\nabla\phi( x) - \nabla\phi( x_0))$ and hence the maximizers $\hat{ x}$ of $g$ over $P$ are characterized by 
\begin{equation*}
    t\langle c, v\rangle = \langle \nabla\phi(\hat{ x}) - \nabla\phi( x_0), v\rangle \quad \text{for all } v\in TP. 
\end{equation*} 
On the other hand, for the gradient flow, we can use~\eqref{eq:explicitFRGF} and compute for $v\in TP$ 
\begin{align*}
    \langle \nabla\phi( x_t) - \nabla\phi( x_0), v \rangle & = \int_{0}^t \partial_s \langle \nabla \phi( x_s), v \rangle \D s 
    \\ & = \int_{0}^t \langle \nabla^2 \phi( x_s)\partial_s  x_s, v \rangle \D s 
    \\ & = \int_{0}^t \langle \nabla f(x_s), v \rangle \D s 
    \\ & = \int_{0}^t \langle c, v \rangle \D s = t \langle c, v \rangle . 
\end{align*}
This shows $x_t$ maximizes $g$ over $P$ as claimed. 
\end{proof}

\begin{corollary}[Sublinear convergence]\label{cor:sublinearRate}
    Consider Setting~\ref{setting:generalLP} and denote the face of maximizers of the linear program~\eqref{eq:LP} by $F^\star$ and fix $x^\star \in \argmin_{ x\in F^\star} D_\phi ( x,  x_0)$. 
    Then it holds that 
    \begin{equation}\label{eq:sublinearConvergence}
        c^\top  x^\star - c^\top x_t \le \frac{D_\phi ( x^\star,  x_0) - D_\phi ( x_t,  x_0)}t \le \frac{D_\phi ( x^\star,  x_0)}t \quad \text{for all } t\in[0, T). 
    \end{equation}
\end{corollary}
\begin{proof}
    By the central path property, we have 
    \begin{equation*}
        c^\top  x_t - t^{-1} D_\phi ( x_t,  x_0) \ge c^\top  x^\star - t^{-1} D_\phi ( x^\star,  x_0).
    \end{equation*}
    Rearranging yields the result. 
\end{proof}

\begin{corollary}[Implicit bias of Hessian gradient flows of LPs]\label{cor:implicitBias}
    Consider Setting~\ref{setting:generalLP} and denote the face of maximizers of the linear program~\eqref{eq:LP} by $F^\star$, assume that $\phi$ is strictly convex and continuous on its domain, and assume that the Hessian gradient flow $(x_t)_{t\ge0}$ exists for all times. 
    Then it holds that 
    \begin{equation}
        \lim_{t\to+\infty} x_t =  x^\star = \argmin_{ x\in F^\star} D_\phi ( x,  x_0). 
    \end{equation}
    In words, the Hessian gradient flow converges to the Bregman projection of $ x_0$ to $F^\star$. 
\end{corollary}
\begin{proof} 
By compactness of $P$, the sequence $( x_{t_n})_{n\in\N}$ has at least one accumulation point for any sequence $t_n\to+\infty$. 
Hence, we can assume without loss of generality that $ x_{t_n}\to\hat x$ and it remains to identify $\hat x$ as the information projection $ x^\star\in F^\star$. 

Surely, we have $\hat x\in F^\star$ as $c^\top \hat x = \lim_{n\to\infty} c^\top  x_{t_n} = \max_{ x\in P} c^\top  x$ by \Cref{cor:sublinearRate}. 
Further, by the central path property, we have for any optimizer $ x'\in F^\star$ that 
\begin{equation*}
    c^\top  x_t - t^{-1} D_\phi ( x_t,  x_0) \ge 
    c^\top x' - t^{-1} D_\phi ( x',  x_0) 
\end{equation*}
and therefore 
\begin{equation*}
    D_\phi ( x',  x_0)  - D_\phi ( x_t,  x_0) \ge t c^\top ( x'- x_t) \ge0.  
\end{equation*}
Hence, we have  
\begin{align*}
    D_\phi (\hat x,  x_0) = \lim_{n\to\infty} D_\phi ( x_{t_n},  x_0)  \le D_\phi ( x',  x_0)
\end{align*}
and can conclude by minimizing over $ x'\in F^\star$. 
\end{proof}

\section{Geometry and Sublinear Convergence of Kakade Gradient Flows}\label{sec:geometry} 
Our goal is to study the solutions $\pi_\tau^\star$ of the entropy-regularized reward $R_\tau$. 
For linear programs, we have seen in \Cref{subsec:H-GF} that the solutions of the regularized problems solve the corresponding Hessian gradient flow. 
Recall that the entropy-regularized reward optimization problem is equivalent to a linear program in state-action space with a conditional KL regularization. 
Hence, the state-action distributions solve a Hessian gradient flow and consequently, the optimal regularized policies solve a gradient flow with respect to some metric. 
In this section, we study this Riemannian metric on the space of policies and provide an explicit expression of the gradient dynamics. 

\subsection{The Kakade metric and policy gradient theorems} 

The following metric on the policy space was proposed in the context of natural gradients~\cite{kakade2001natural}. 

\begin{definition}[Kakade metric]
    We call the Riemannian metric $g^\K$ on $\operatorname{int}(\Delta_\A^\bS)$ defined by 
    \begin{equation}
        g_\pi^\K(w_1,w_2) \coloneqq \sum_{s\in\bS} d^\pi(s) \sum_{a\in\A} \frac{w_1(s,a)w_2(s,a)}{\pi(a|s)} \quad \text{for all } w_1, w_2\in T \Delta_\A^\bS
    \end{equation}
    the \emph{Kakade metric}. 
    For differentiable  $f\colon\Delta_\A^\bS\to \R$, we denote the Riemannian gradient by $\nabla^\K f(\pi)$. 
\end{definition}

The Kakade metric has been referred to as the {Fisher-Rao metric} on $\Delta_\A^\bS$ as this reduces to the Fisher-Rao metric if $\lvert \bS\rvert=1$, see~\cite{kerimkulov2023fisher}. 
We choose the name Kakade metric here, as there exist multiple extensions of the Fisher-Rao metric to products of simplices. 
For example, in a game-theoretic context, when considering independently chosen strategies of the players, it might be more natural to work with the product metric, meaning the sum of the Fisher-Rao metrics over the individual factors, which corresponds to the pullback of the Fisher-Rao metric on the simplex of joint distributions under the independence model~\cite{montufar2014fisher,boll2024geometric}. 
Other weightings of the Fisher-Rao metrics over the individual factors are also possible, see~\cite{montufar2014fisher} for an in-depth discussion of different choices and their invariance properties. 
Further, this specific Riemannian metric has been described as the limit of weighted Fisher-Rao metrics on the finite-horizon path spaces~\cite{bagnell2003covariant, peters2003reinforcement, wolfer2023information}. 
Note that although the Kakade metric is closely connected to the weighted entropy $H_\K$, it is not the Hessian metric of $H_\K$, in fact, it is not a Hessian metric at all~\cite[Remark 13]{muller2024geometry}. 

In general, the Kakade metric is only a pseudo-metric. The following assumption ensures positive definiteness, which we assume for the remainder of our analysis. 

\begin{asu}[State exploration]\label{asu:exploration}
For any policy $\pi\in\operatorname{int}(\Delta_\A^\bS)$ the discounted state distribution is positive, meaning that $d^\pi(s)>0$ for all $s\in\bS$. 
\end{asu}

Kakade gradient flows are well-posed, meaning that they admit a unique solution $(\pi_t)_{t\in\R_{\ge0}}$, both in the unregularized case~\cite{alvarez2004hessian, muller2024geometry} and the regularized case~\cite{mueller2023thesis,kerimkulov2023fisher}. 

The \emph{policy gradient theorem} states that 
\begin{align}
\partial_{\theta_i} R(\pi) = \frac1{1-\gamma} \sum_{s\in\bS} d^{\pi_\theta}(s) \sum_{a\in\A} \partial_{\theta_i} \pi_\theta(a|s) A^{\pi_\theta}(s,a)
\end{align}
and connects the gradient of the reward to the advantage function~\cite{sutton1999policy,agarwal2021theory}. 
Inspired by this, we provide an explicit formula for the gradient of the reward with respect to the Kakade metric. 

\begin{theorem}[Gradient with respect to the Kakade metric]\label{thm:PGtheorem}
Let \Cref{asu:exploration} hold. 
Then for all $\pi\in\operatorname{int}(\Delta_\A^\bS)$ and $a\in\A$, $s\in\bS$ it holds that 
\begin{equation}\label{eq:PG}
\nabla^\K R(\pi)(s,a) = (1-\gamma)^{-1}A^{\pi}(s,a) \pi(a|s). 
\end{equation}
\end{theorem}

In the proof, we use the following auxiliary result. 

\begin{lemma}[Derivatives of state-action distributions, \cite{muller2022pomdps}]\label{lem:derivativeSAD}
It holds that 
\begin{align}
    \frac{\partial\nu^\pi}{\partial\pi(a|s)} = d^\pi(s)(I-\gamma P_\pi^T)^{-1} e_{(s,a)},
\end{align}
where $P_\pi\in\Delta_{\bS\times\A}^{\bS\times\A}$ is given by $P_\pi(s', a'|s,a) = \pi(a'|s')P(s'|s,a)$. 
\end{lemma}

\begin{proof}[Proof of \Cref{thm:PGtheorem}]
First, note that $A^\pi\odot\pi\in T\Delta_\A^\bS$ where $(A^\pi\odot\pi)(s,a) = A^\pi(s,a)\pi(a|s)$ and that $R$ and $g^\K$ can be extended to a neighborhood of $\Delta_\A^\bS$. 
It suffices to show 
\begin{align*}
    (1-\gamma)^{-1} g_\pi^\K(A^\pi\odot\pi, w) = \partial_w R(\pi) 
    \quad \text{for all } w\in T\Delta_\A^\bS. 
\end{align*}
Since $R(\pi) = r^\top \nu^\pi$, we have $\frac{\partial R(\pi)}{\partial\pi(a|s)} = r^\top \frac{\partial\nu^\pi}{\partial\pi(a|s)}$.
For any tangent vector $w\in T\Delta_\A^\bS$ we can use \Cref{lem:derivativeSAD} to compute 
\begin{align*}
    \partial_w R(\pi) & 
    = \sum_{s\in\bS, a\in\A} w(s,a) \cdot\frac{\partial R(\pi)}{\partial\pi(a|s)}  \\ & 
    = \sum_{s\in\bS, a\in\A} d^\pi(s) w(s,a)\left\langle (I-\gamma P_\pi^T)^{-1} e_{(s,a)}, r \right\rangle \\ & 
    = \sum_{s\in\bS, a\in\A} d^\pi(s) w(s,a)\left\langle  e_{(s,a)}, (I-\gamma P_\pi)^{-1}r \right\rangle \\ & 
    = (1-\gamma)^{-1}\sum_{s\in\bS, a\in\A} d^\pi(s) w(s,a) Q^\pi(s,a) \\ & 
    = (1-\gamma)^{-1}\sum_{s\in\bS, a\in\A} d^\pi(s) w(s,a) (Q^\pi(s,a) - V^\pi(s)) \\ & 
    = (1-\gamma)^{-1} g_\pi^\K(A^\pi\odot \pi, w),
\end{align*}
where we have used the Bellman equation $Q^\pi = (1-\gamma)(I-\gamma P_\pi)^{-1}r$ as well as $\sum_{a\in\A}w(s,a) = 0$, which holds for tangent vectors $w\in T\Delta_\A^\bS$. 
\end{proof}

Note that the gradient of the reward with respect to the Kakade metric is independent of the initial distribution $\mu\in\Delta_\bS$, both for the regularized and the unregularized reward. 

\begin{definition}[Kakade gradient flow]
We say that $(\pi_t)_{t\in[0,T)}\subseteq \operatorname{int}(\Delta_\A^\bS)$ solves the \emph{Kakade gradient flow} if 
\begin{align}\label{eq:K-GF}
    \partial_t \pi_t = \nabla^\K R(\pi_t).
\end{align}
\end{definition}

By \Cref{thm:PGtheorem}, a solution of the Kakade gradient flow satisfies 
\begin{align}\label{eq:explicitExpressionKGF}
    \partial_t \pi_t(a|s) = (1-\gamma)^{-1} A^{\pi_t}(s,a)\pi_t(a|s). 
\end{align}
This explicit expression allows an explicit analysis of the gradient flow and thus entropy regularization. 

\subsection{State-action geometry of Kakade gradient flows} 
Our goal is to use tools from Hessian gradient flows, but the Kakade metric is not a Hessian metric. 
However, it was shown that the policy polytope $\Delta_\A^\bS$ endowed with the Kakade metric is isometric to the state-action polytope endowed with the Hessian metric induced by the conditional entropy~\cite{muller2024geometry}. 
This allows us to borrow from the results on Hessian gradient flows. 

\begin{definition}[Conditional Fisher-Rao metric]
We call the metric $g^{A|S}$ on $\operatorname{int}(\cD)$ given by 
\begin{align}
    g^{A|S}_\nu(w_1,w_2) \coloneqq w_1^\top \nabla^2 H_{A|S}(\nu) w_2 \quad \text{for } w_1, w_2\in T\cD 
\end{align}
the \emph{conditional Fisher-Rao metric} and the denote the corresponding gradient by $\nabla^{A|S}$. 
\end{definition}

It is elementary to check the convexity of $H_{A|S}$ on $\R_{>0}^{\A\times\bS}$, but it also easily seen that $\nabla^2 H_{A|S}$ has zero eigenvalues. 
However, it can be shown that $\nabla^2 H_{A|S}$ is strictly definite on $T\cD$ and therefore induces a Riemannian metric on $\operatorname{int}(\cD)$, see~\cite{muller2024geometry}. 

The following result relates the Kakade metric and Kakade divergence to the Hessian metric and Bregman divergence of the conditional entropy $H_{A|S}$. 

\begin{theorem}[State-action geometry of the Kakade metric, \cite{{muller2024geometry}}]\label{thm:SAGeometryKakade}
Consider a finite Markov decision process and let \Cref{asu:exploration} hold. 
Then the mapping
\begin{align}
    \begin{split}\label{eq:isometry}
        \Xi\colon(\operatorname{int}(\Delta_\A^\bS), g^\K) & \to (\operatorname{int}(\cD), g^{A|S}) \\ 
        \pi & \mapsto \nu^\pi \\ \nu(\cdot|\cdot) & \; \leftmapsto \; \nu 
    \end{split}
\end{align}
between policies and state-action distributions is an isometric diffeomorphism, where the conditioning map is defined in~\eqref{eq:conditioning}.
On the whole conditional probability polytope, $\Xi\colon \Delta_\A^\bS \to \cD$, $\pi\mapsto \nu^\pi$ is a continuous bijection with continuous inverse.
\end{theorem}

As isometries map gradient flows to gradient flows, \Cref{thm:SAGeometryKakade} implies that $(\pi_t)_{t\ge0}$ solves the Kakade gradient flow $\partial_t R(\pi_t)= \nabla^\K R(\pi_t)$ if and only if $(\nu_t)_{t\ge0}=(\nu^{\pi_t})_{t\ge0}$ solves the conditional Fisher-Rao gradient flow $\partial_t \nu_t = \nabla^{A|S} f(\nu_t)$ if the following diagram commutes 
\begin{equation}
    \begin{tikzcd}
        \Delta_\A^\bS \arrow[r, "\Xi"] \arrow[rd, "R" below left] & \cD \arrow[d,"f"] \\ 
        & \mathbb R 
    \end{tikzcd}, \quad \text{where }
    \begin{tikzcd}
        \pi \arrow[r, mapsto] \arrow[rd, mapsto] & \nu^\pi \arrow[d, mapsto] \\ 
        & R(\pi). 
    \end{tikzcd}
\end{equation}
The isometry property allows us to transfer the theory on Hessian gradient flows of linear programs to Kakade gradient flows despite the Kakade metric not being Hessian. 

\begin{lemma}[Isometries preserve gradient flows]\label{lem:isometries}
    Consider an isometric diffeomorphism 
    \begin{align}
        \Xi \colon (\mathcal M, g^{\mathcal M}) \to (\mathcal N, g^{\mathcal N})
    \end{align}
    between two Riemannian manifolds and a differentiable function $g\colon \mathcal N\to\mathbb R$ and set $f\coloneqq g\circ\Xi$. Further, we fix an open interval $I\subseteq\mathbb R$ and consider two differentiable curves $x\colon I\to\mathcal M$ and $y = \Xi\circ x\colon I\to\mathcal N$. 
    Then the following statements are equivalent: 
    \begin{enumerate}
        \item The curve $x\colon I\to\mathcal M$ is a gradient flow of $f$, meaning that 
        \begin{align}\label{eq:GFx}
            g^{\mathcal M}_{x_t}(\partial_t x_t, v) = df(x_t) v \quad \text{for all } v\in T_{x_t} \mathcal M, t\in I. 
        \end{align}
        \item The curve $y\colon I\to\mathcal N$ is a gradient flow of $g$, meaning that 
        \begin{align}
            g^{\mathcal N}_{y_t}(\partial_t y_t, w) = dg(y_t) w \quad \text{for all } w\in T_{y_t} \mathcal N, t\in I. 
        \end{align}
    \end{enumerate} 
\end{lemma}
\begin{proof}
Assume first that $x$ satisfies the gradient flow equation~\eqref{eq:GFx}. 
As $\Phi$ is a diffeomorphism, for any $w\in T_{y_t} \mathcal N$ there is a tangent vector $v\in T_{y_t} \mathcal M$ such that $d\Xi(x_t)v=w$. 
Using the isometry property and the gradient flow equation, we check
\begin{align*}
    g^{\mathcal N}_{y_t}(\partial_t y_t, w) = g^{\mathcal N}_{\Xi(x_t)}(d\Xi(x_t)\partial_t x_t, d\Xi(x_t)v) = g^{\mathcal M}_{x_t}(\partial_t x_t, v) = df(x_t)v = dg(\Xi(x_t)) d\Xi(x_t) v = dg(y_t) w. 
\end{align*}
The equivalence follows from symmetry.
\end{proof}

\begin{theorem}[Implications from Hessian gradient flows]\label{cor:implications}
Let \Cref{asu:exploration} hold, denote the set of optimal policies by  $\Pi^\star\coloneqq\{ \pi\in\Delta_\A^\bS : R(\pi) = R^\star \}$, and fix an initial policy $\pi_0\in\operatorname{int}(\Delta_\A^\bS)$. 
Then the following statements hold: 
\begin{enumerate}
    \item \emph{Well-posedness:}\label{item:wellposedness} The Kakade gradient flow~\eqref{eq:K-GF} admits a unique global solution $(\pi_t)_{t\ge0}$.  
    \item \emph{Central path property:}\label{item:centralPath} For all $t\ge0$ it holds that 
    \begin{align}
        \pi_t = \argmax_{\pi\in\Delta_\A^\bS} \Big\{R(\pi) - t^{-1} D_\K(\pi,\pi_0)\Big\}. 
    \end{align}
    \item \emph{Sublinear convergence:}\label{item:sublinear}
    For all $t\ge0$ we have 
    \begin{align}\label{eq:sublinearConvergenceReward}
            0\le R^\star - R(\pi_t)\le \frac{\min_{\pi^\star\in\Pi^\star}D_\K(\pi^\star, \pi_0) - D_\K(\pi_t, \pi_0)}{t} \le \frac{\min_{\pi^\star\in\Pi^\star}D_\K(\pi^\star, \pi_0)}{t}.
    \end{align}
    \item \emph{Implicit bias:}\label{item:implicitBias} The gradient flow $(\pi_t)_{t\ge0}$ converges globally and it holds that 
    \begin{align}
        \lim_{t\to+\infty} \pi_t = \pi^\star = \argmin_{\pi\in\Pi^\star} D_\K(\pi, \pi_0). 
    \end{align}    
\end{enumerate}
\end{theorem}
\begin{proof}
We set 
\begin{equation}
    \pi_t = \argmax_{\pi\in\Delta_\A^\bS} \Big\{R(\pi) - t^{-1} D_\K(\pi,\pi_0)\Big\},
\end{equation}
where it is elementary to check $\pi_t\in \operatorname{int}(\Delta_\A^\bS)$. 
By~\Cref{prop:SAGeometryLP}, the element $\nu_t\coloneqq \nu^{\pi_t}$ solves the regularized linear program with regularization strength $t^{-1}$ and by~\Cref{prop:centralPathLP} $(\nu_t)_{t\ge0}$ solves the Hessian gradient flow with respect to the conditional Fisher-Rao metric. 
As $\pi\mapsto \nu^\pi$ is an isometric diffeomorphism by \Cref{thm:SAGeometryKakade}, \Cref{lem:isometries} ensures that $(\pi_t)_{t\ge0}$ solves the Kakade gradient flow. 
This shows both \ref{item:wellposedness}. and \ref{item:centralPath}.. 
For \ref{item:sublinear}., we pick $\pi^\star\in\argmin_{\pi\in\Pi^\star}D_{\K}(\pi^\star, \pi_0)$ and set $\nu^\star\coloneqq\nu^{\pi^\star}$. 
By \Cref{prop:SAGeometryLP} we have use \Cref{cor:sublinearRate} to obtain
\begin{equation*}
    R^\star-R(\pi_t) = r^\top \nu^\star - r^\top \nu_t \le \frac{D_{A|S}(\nu^\star, \nu_0) - D_{A|S}(\nu_t, \nu_0)}{t} = \frac{D_{\K}(\pi^\star, \pi_0) - D_{\K}(\pi_t, \pi_0)}{t}. 
\end{equation*}
Finally, for \ref{item:implicitBias}. we use \Cref{cor:implicitBias} together with the continuity of the inverse of the state-action map $\Xi$ and obtain 
\begin{equation*}
    \lim_{t\to+\infty} \pi_t = \Xi^{-1}(\lim_{t\to+\infty} \nu_t) = \Xi^{-1}(\nu^\star) = \pi^\star. 
\end{equation*}
\end{proof}

\begin{remark}[Implicit and algorithmic bias]
In the case of multiple optimal policies, the Kakade gradient flow will converge towards the Kakade projection of the initial policy to the set of optimal policies. 
The selection of gradient schemes of a particular optimizer in the case of multiple optima is commonly referred to as the \emph{implicit} or \emph{algorithmic bias} of the method. 
In the context of reinforcement learning, the implicit bias was studied for discrete-time natural policy gradient and policy mirror descent methods. 
First, for a natural actor-critic scheme in linear Markov decision processes the Kakade divergence of the optimization trajectory $(\pi_k)_{k\in\mathbb N}$ to the maximum entropy policy $\pi^\star$ is bounded by $D_\K(\pi^\star, \pi_k)\le \log k + (1-\gamma)^{-2}$~\cite{hu2021actor}. 
This control on the Bregman divergence to the maximum entropy policy ensures that the probability of selecting an optimal action $a\in A^\star_s$ decays at most like $\pi_k(a|s)\ge c k^{-1}$, but fails to identify the limiting policy. 

On the other hand, a policy mirror descent scheme with decaying entropy regularization strength $\tau_k$ was studied in~\cite{li2023homotopic}. 
Here, it is shown that for several choices of regularization strengths and stepsizes, the resulting policies $\pi_k$ converge to the maximum entropy optimal policy, thereby characterizing the algorithmic bias of this approach. 
Whereas this analysis can characterize the limit of the optimization scheme, it utilizes decaying explicit regularization. 

In contrast, our implicit bias result works in continuous time and utilizes the correspondence between the gradient flows and regularized problems. 
This allows us to show convergence towards the (generalized) maximum entropy optimal policies without relying on explicit regularization.  

In other contexts, in particular for regression tasks, a huge variety of implicit bias results have been established for gradient and mirror descent~\cite{soudry2018implicit, gunasekar2018characterizing, chizat2020implicit, bowman2022implicit, jin2023implicit, cohen2023deep, chou2024gradient}. 
In particular, in~\cite{pmlr-v125-ji20a}, the dynamics of gradient descent are shown to stay close to the solutions of norm-constrained optimization problems for a linear regression problem. 
This allows them to show convergence in the direction of the gradient descent iterates towards the maximum-margin solution. 
In contrast, we study the reward optimization problem and natural policy gradient flows, which allows us to obtain a precise characterization of the optimization trajectory via the central path property. 
\end{remark}

Where \Cref{cor:implications} provides sublinear convergence with respect to the reward function, we use this to show sublinear convergence of the advantage function. 

\begin{proposition}[Sublinear convergence of advantage functions]\label{cor:convValueFunction}
    Let \Cref{asu:exploration} hold, denote the set of optimal policies by  $\Pi^\star\coloneqq\{ \pi\in\Delta_\A^\bS : R(\pi) = R^\star \}$, and fix an initial policy $\pi_0\in\Delta_\A^\bS$ and denote the generalized maximum entropy optimal policy by $\pi^\star = \argmin_{\pi\in\Pi^\star} D_\K(\pi, \pi_0)$. 
    Then for all $t\ge0$, it holds that 
    \begin{align}\label{eq:sublinearConvergenceVF}
        0\le V^\star(s) - V^{\pi_t}(s)\le \frac{D_\K(\pi^\star, \pi_0)}{t} \quad \text{for all } s\in \bS. 
    \end{align}
    In particular for any $s\in\bS$ and $a\in\A$, it holds that 
    \begin{align}
        A^{\pi_t}(s,a) & \le A^\star(s,a) + \frac{D_\K(\pi^\star, \pi_0)}{t} 
        \quad \text{and }\\ 
        A^{\pi_t}(s,a) & \ge A^\star(s,a) - \frac{\gamma D_\K(\pi^\star, \pi_0)}{t}.  
    \end{align}
\end{proposition}
\begin{proof}
Let us fix $s\in\bS$. 
By \Cref{thm:PGtheorem}, the Kakade gradient flow $(\pi_t)_{t\ge0}$ is independent of the initial distribution $\mu$, as long as it has full support. 
Let us pick some $\mu_n$ with full support and $\mu_n\to\delta_s$, then $R^{\mu_n}(\pi) \to V^\pi(s)$ for any $\pi$. 
By \Cref{cor:sublinearRate} we have 
    \[ 0\le (R^{\mu_n})^\star - R^{\mu_n}(\pi_t)\le \frac{D_\K(\pi^\star, \pi_0)}{t} \quad \text{for all } t\ge0, n\in\N, \]
which yields \eqref{eq:sublinearConvergenceVF} for $n\to+\infty$. 
We use this to estimate 
\begin{align*}
    Q^\star(s,a) - Q^{\pi_t}(s,a) & = r(s,a) + \gamma \sum_{s'} V^\star(s')  P(s'|s,a) - r(s,a) - \gamma \sum_{s'} V^{\pi_t}(s') P(s'|s,a)
    \\ & = \gamma \sum_{s'} (V^\star(s') - V^{\pi_t}(s')) P(s'|s,a)
    \\ & \le \frac{\gamma D_\K(\pi^\star, \pi_0)}{t}.
\end{align*}
In particular, this implies 
\begin{align*}
    A^\star(s,a) - A^{\pi_t}(s,a) =  Q^\star(s,a) - V^\star(s)  - Q^{\pi_t}(s,a) +  V^{\pi_t}(s) \ge V^{\pi_t}(s) - V^\star(s) \ge - \frac{D_\K(\pi^\star, \pi_0) }{t}
\end{align*}
and similarly
\begin{align*}
    A^\star(s,a) - A^{\pi_t}(s,a) \le Q^\star(s,a) - Q^{\pi_t}(s,a) \le\frac{\gamma D_\K(\pi^\star, \pi_0) }{t}.
\end{align*}
\end{proof}

\section{Optimal Exponential Convergence Rates for Entropy Regularization}\label{sec:sharpAnalysis} 

We now improve the sublinear convergence guarantee from \Cref{prop:generalEstimate} on the entropy regularization error. 
To this end, we work with the interpretation of the solutions of the regularized problems as solutions of the Kakade gradient flow and employ the explicit expression~\eqref{eq:explicitExpressionKGF} and the sublinear convergence guarantee from \Cref{cor:convValueFunction} to establish linear convergence of Kakade gradient flows. 
We complement this with a lower bound that matches the upper bound up to a polynomial factor. 
We work in the following setting. 

\begin{setting}\label{set:linear}
Consider a finite discounted Markov decision process $(\bS, \A, P, \gamma, r)$, an initial distribution $\mu\in\Delta_\bS$ and fix a policy $\pi_0\in\operatorname{int}(\Delta_\A^\bS)$, assume that \Cref{asu:exploration} holds and denote the optimal reward by $R^\star\coloneqq\max\{R(\pi):\pi\in\Delta_\A^\bS\}$. 
Further, let $(\pi_t)_{t\ge0}$ be the unique global solution of the Kakade policy gradient flow~\eqref{eq:K-GF} or equivalently the solutions of the entropy-regularized problems 
\begin{align}
    \pi_t = \argmax_{\pi\in\Delta_\A^\bS} \Big\{R(\pi) - t^{-1} D_\K(\pi,\pi_0)\Big\}. 
\end{align}
We denote the generalized maximum entropy optimal policy by $\pi^\star = \argmin_{\pi\in\Pi^\star} D_\K(\pi, \pi_0)$, where $\Pi^\star\coloneqq\{ \pi\in\Delta_\A^\bS : R(\pi) = R^\star \}$ denotes the set of optimal policies.   
Finally, we set 
\begin{align}
    \Delta \coloneqq -(1-\gamma)^{-1}\max\left\{ A^\star(s,a) : A^\star(s,a)\ne0, s\in\bS, a\in\A \right\}. 
\end{align}
\end{setting}

Note that $\Delta>0$ unless every action is optimal in every state. 
We can interpret $\Delta$ as the minimal suboptimality of a suboptimal action under $A^\star$. 

\subsection{Convergence in value}
First, we study the entropy regularization error in the objective function, meaning that we study the reward achieved by the optimal regularized policies. 

\begin{theorem}[Convergence in value]\label{thm:convergenceValue}
Consider \Cref{set:linear} and set $c \coloneqq (1-\gamma)^{-1}D_\K(\pi^\star, \pi_0)$. 
For any $t\ge1$ it holds that 
\begin{align}
    R^\star-R(\pi_t) & \le \frac{2\lVert r \rVert_\infty}{1-\gamma} \cdot e^{-\Delta (t-1) + c \log t} \quad \text{as well as}
    \label{eq:convergenceValueUpper}
    \\ R^\star-R(\pi_t) & \ge \Delta \cdot \left(\min\limits_{s\in\bS} d^{\pi_t}(s)\sum_{a\notin A_s^\star}\pi_0(a|s) \right) \cdot e^{-\Delta (t-1) - \gamma c \log t - 2 \lVert r \rVert_\infty} .
\end{align} 
\end{theorem}

Note that the coefficient of the lower bound depends on $t$. 
However, the coefficient does not become arbitrarily small as $\min_{s} d^{\pi_t}(s)\to \min_{s} d^{\pi^\star}(s)>0$ for $t\to +\infty$. 
Further, the policies $(\pi_t)_{t\ge0}$ do not depend on $\mu$, where $d^{\pi_t}$ does. 
If we choose $\mu$ to be the uniform distribution, then $d^\pi(s)\ge (1-\gamma)\lvert \bS \rvert^{-1}$, which yields a lower bound of 
\begin{align*}
    R^\star-R(\pi_t) 
    \ge 
    (1-\gamma)\Delta{\lvert \bS \rvert}^{-1}  \sum_{a\notin A_s^\star}\pi_0(a|s) \cdot e^{-\Delta (t-1) - \gamma c \log t - 2 \lVert r \rVert_\infty},
\end{align*}
where the coefficient is independent of $t$.

The above result ensures that the probability $\pi_t(a|s)$ of selecting a suboptimal action $a$ decays exponentially, which implies exponential convergence of the reward $R(\pi_t)$ achieved by the policies. 

We combine the explicit expression~\eqref{eq:explicitExpressionKGF} as well as the sublinear convergence of the advantage function towards the optimal advantage function $A^\star$ to bound the individual entries of the policies along the gradient flow trajectory. 

\begin{lemma}
\label{lem:convergenceProbabilities}
Consider \Cref{set:linear}. 
Then for all $t\ge t_0>0$, $s\in\bS$, and $a\in\A$ it holds that 
\begin{align}
    \pi_t(a|s) & \le \pi_{t_0}(a|s) \exp\left(\frac{A^\star(s,a)(t-t_0) + D_\K(\pi^\star, \pi_0) \log (\frac{t}{t_0})}{1-\gamma}\right) \\ 
    \pi_t(a|s) & \ge 
    \pi_{t_0}(a|s) \exp\left(\frac{A^\star(s,a)(t-t_0) - \gamma  D_\K(\pi^\star, \pi_0)\log (\frac{t}{t_0}) }{1-\gamma}\right)
\end{align}
and for $t\ge1$ it holds that 
\begin{align}
     \pi_t(a|s) & \le \pi_0(a|s) \exp\left(\frac{A^\star(s,a)(t-1) + D_\K(\pi^\star, \pi_0) \log {t} + 2\lVert r \rVert_\infty }{1-\gamma}\right) \\ 
     \pi_t(a|s) & \ge \pi_0(a|s) \exp\left(\frac{A^\star(s,a)(t-1) - \gamma D_\K(\pi^\star, \pi_0) \log t - 2\lVert r \rVert_\infty}{1-\gamma}\right).
\end{align}
\end{lemma}
\begin{proof}
As the policies $(\pi_t)_{t\ge0}$ solve the Kakade policy gradient flow in $\Delta_\A^\bS$ we have 
\begin{equation}
    \partial_t \pi_t(a|s) = (1-\gamma)^{-1} A^{\pi_t}(s,a) \pi_t(a|s).
\end{equation}
By \Cref{cor:convValueFunction} it holds that 
\begin{equation*}
    (1-\gamma)^{-1} \left(A^{\star}(s,a) - \frac{\gamma D_\K(\pi^\star, \pi_0)}t\right)\pi_t(a|s) \le \partial_t \pi_t(a|s) \le (1-\gamma)^{-1} \left(A^{\star}(s,a) + \frac{ D_\K(\pi^\star, \pi_0)}{t}\right)\pi_t(a|s).
\end{equation*}
Now, Gr\"{o}nwall's inequality yields 
\begin{align*}
    \pi_T(a|s) & \le \pi_{t_0}(a|s) \exp\left(\frac{A^{\star}(s,a) (T-t_0) + D_\K(\pi^\star, \pi_0) \int_{t_0}^T t^{-1}\mathrm dt}{1-\gamma} \right) 
    \\ & 
    = \pi_{t_0}(a|s) \exp\left(\frac{A^{\star}(s,a) (T-t_0) + D_\K(\pi^\star, \pi_0) \log T - D_\K(\pi^\star, \pi_0) \log t_0}{1-\gamma}\right) 
\end{align*}
as well as 
\begin{align*}
    \pi_T(a|s) & \ge \pi_{t_0}(a|s) \exp\left(\frac{A^{\star}(s,a)(T-t_0) - \gamma D_\K(\pi^\star, \pi_0) \int_{t_0}^T t^{-1}\mathrm{d}t}{1-\gamma} \right) 
    \\ & 
    = \pi_{t_0}(a|s) \exp\left( \frac{A^{\star}(s,a) (T-t_0) - \gamma D_\K(\pi^\star, \pi_0) \log T + \gamma D_\K(\pi^\star, \pi_0) \log t_0}{1-\gamma} \right) . 
\end{align*}
Further, note that $\lVert A^\pi \rVert_\infty \le \lVert Q^\pi \rVert_\infty + \lVert V^\pi \rVert_\infty \le 2 \lVert r \rVert_\infty$ and hence Gr\"{o}nwall yields 
\begin{align*}
    \pi_{0}(a|s)e^{-\frac{2 \rVert r \rVert_\infty t_0}{1-\gamma}}\le\pi_{t_0}(a|s) \le \pi_{0}(a|s)e^{\frac{2 \rVert r \rVert_\infty t_0}{1-\gamma}}
\end{align*}
and choosing $t_0=1$ finishes the proof.  
\end{proof}

\begin{lemma}[Sub-optimality gap]\label{lem:subOptimalityGap}
Consider a discrete discounted Markov decision process. 
Then, for any policy $\pi\in\Delta_\A^\bS$ it holds that
    \begin{align}
        \Delta \cdot \min_{s\in\bS} \sum_{s\in\bS,a\notin A_s^\star} d^\pi(s)\pi(a|s) \le 
        R^\star-R(\pi) \le \frac{2\lVert r \rVert_\infty}{1-\gamma} \sum_{s\in\bS,a\notin A_s^\star} d^\pi(s) \pi(a|s),
    \end{align}
    where $A_s\coloneqq \{ a\in\A : A^\star(s,a) = 0 \}$ denotes the set of optimal actions in $s\in\bS$. 
\end{lemma}
\begin{proof}
By the performance difference \Cref{lem:PD}, we have
\begin{align*}
    (1-\gamma)(R^\star-R(\pi)) & = - \sum_{s\in\bS,a\notin A_s^\star} d^{\pi}(s) \pi(a|s) A^\star(s,a)
    \le \lVert A^\star \rVert_\infty \sum_{s\in\bS,a\notin A_s^\star} d^{\pi}(s) \pi(a|s) 
\end{align*}
Further, note that $\lvert A^\star(s,a)\rvert\le \lvert Q^\star(s,a)\rvert+\lvert V^\star(s)\rvert\le 2\lVert r \rVert_\infty$. 
The lower bound follows with an analog argument as 
\begin{align*}
    (1-\gamma)(R^\star-R(\pi)) & = - \sum_{s\in\bS,a\notin A_s^\star} d^{\pi}(s) \pi(a|s) A^\star(s,a)
    \ge (1-\gamma)\Delta\sum_{s\in\bS,a\notin A_s^\star} d^{\pi}(s) \pi(a|s). 
\end{align*}
\end{proof}

\begin{proof}[Proof of \Cref{thm:convergenceValue}]
We use \Cref{lem:subOptimalityGap} together with \Cref{lem:convergenceProbabilities} and estimate 
\begin{align*}
    \sum_{s\in\bS,a\notin A_s^\star} d^{\pi_t}(s) \pi_t(a|s) 
    & \le 
    \sum_{s\in\bS,a\notin A_s^\star} d^{\pi_t}(s) \pi_1(a|s) e^{-\Delta (t-1) + \gamma c \log t} \le e^{-\Delta (t-1) + \gamma c \log t},
\end{align*} 
which yields~\eqref{eq:convergenceValueUpper}. 
For the lower bound, we fix $s_0\in \bS$ and $a_0\in \A$ with $(1-\gamma)^{-1}A^\star(s_0,a_0)=-\Delta$ and estimate 
\begin{align*}
    \sum_{s\in\bS,a\notin A_s^\star} d^{\pi_t}(s)\pi_t(a|s) 
    & \ge 
    \sum_{s\in\bS,a\notin A_s^\star} d^{\pi_t}(s)\pi_0(a|s) e^{(1-\gamma)^{-1}A^\star(s,a)(t-1)  - c\log t - 2 \lVert r \rVert_\infty}
    \\ & 
    \ge d^{\pi_t}(s_0)\sum_{a\notin A^\star_{s_0}}\pi_0(a|s_0) e^{-\Delta (t-1)  - c\log t - 2 \lVert r \rVert_\infty}
    \\ & 
    \ge \min_{s\in\bS} \left\{d^{\pi_t}(s)\sum_{a\notin A_s^\star}\pi_0(a|s) \right\} \cdot e^{-\Delta (t-1)  -  c\log t - 2 \lVert r \rVert_\infty}
\end{align*}
\end{proof}

\subsection{Convergence of policies}
Having studied the decay of the suboptimality gap $R^\star - R(\pi_t)$, we now give analogous bounds for the convergence of the policies measured in the Kakade divergence. 

\begin{theorem}[Convergence of policies]\label{thm:convergencePolicies}
Consider \Cref{set:linear} and set $c \coloneqq (1-\gamma)^{-1}D_\K(\pi^\star, \pi_0)$. 
For any $t\ge1$, it holds that  
    \begin{align}
        D_\K(\pi^\star, \pi_t) 
        & \le 
        \frac{e^{-\Delta (t-1) + c\log t}}{1-e^{-\Delta (t-1) + c\log t}} \quad \text{as well as }
        \\ D_\K(\pi^\star, \pi_t) 
        & \ge \min_{s\in \bS, a\notin A_s^\star} d^\star(s)\pi_0(a|s) \cdot e^{-\Delta(t-1)- \gamma c \log t - 2 \lVert r \rVert_\infty}, 
    \end{align}
    where the upper bound holds if $e^{-\Delta (t-1) + c\log t}<1$ which is satisfied for $t>0$ large enough. 
\end{theorem}

Note that the denominator of the upper bound converges to $1$ for $t\to+\infty$. 
We will use the following auxiliary result in the proof of \Cref{thm:convergencePolicies}. 

\begin{lemma}[Information projection onto faces]\label{lem:KLProperties}
    Consider a finite set $\X$ and the face $F\coloneqq \{ \mu\in\Delta_\X : \mu_x = 0 \text{ for all } x\notin X \}$ of the simplex $\Delta_\X$ for some $X\subseteq\X$. 
    Then 
    \begin{align}
        \min_{\mu\in F} \KL(\mu, \nu) = - \log\left(\sum_{x\in X} \nu_{x}\right). 
    \end{align}
    In particular, for any $\mu, \nu\in\Delta_\X$ we have 
    \begin{align}\label{eq:lowerBoundKL}
        \KL(\mu, \nu) \ge -\log\left( \sum_{x\in\operatorname{supp}(\mu)} \nu_x \right). 
    \end{align}
\end{lemma}
\begin{proof}
    We set $g(\mu)\coloneqq \KL(\mu, \nu)$ and consider the information projection $\hat\nu=\argmin_{\mu\in F}\KL(\mu, \nu)$ of $\nu\in\Delta_\X$ onto $F$, which is characterized by 
    \begin{align*}
        v^\top \nabla g(\hat\nu) = 0 \quad \text{for all } v\in TF = \operatorname{span}\{ \nu_{x_1} - \nu_{x_2} : x_1, x_2\in X \}. 
    \end{align*} 
    As $\partial_{x} g(\mu) = \log(\frac{\mu_x}{\nu_x}) + 1$ this is equivalent to 
    \begin{align*}
        \log\left(\frac{\hat\nu_{x_1}}{\nu_{x_1}}\right) = \log\left(\frac{\hat\nu_{x_2}}{\nu_{x_2}}\right) \quad \text{for all } x_1, x_2\in X. 
    \end{align*}
    This implies $\frac{\hat\nu_{x}}{\hat\nu_{x}}$ is constant for $x\in X$ and hence we obtain 
    \begin{align*}
        \hat\nu_x = \begin{cases}
            \frac{\nu_x}{\sum_{x'\in X} \nu_{x'}} \quad & \text{for } x \in X \\ 
            0 & \text{for } x\notin X. 
        \end{cases}
    \end{align*}
    Setting $c^{-1}\coloneqq \sum_{x\in X} \nu_x$ we obtain 
    \begin{align*}
        \min_{\mu\in F} \KL(\mu, \nu) = \KL(\hat\nu, \nu) & = \sum_{x\in X} c\nu_x \log\left(\frac{c\nu_x}{\nu_x}\right) = \log c. 
    \end{align*}
    To show~\eqref{eq:lowerBoundKL}, we choose $X\coloneqq \operatorname{supp}(\mu)$ and use $\KL(\mu, \nu)\ge \min_{\xi\in F}\KL(\xi, \nu)$. 
\end{proof}

\begin{lemma}\label{lem:boundingKakade}
Consider a finite MDP and let $\pi^\star\in\Pi^\star$ be the Kakade projection of $\pi\in\operatorname{int}(\Delta_\A^\bS)$ onto the set of optimal policies $\Pi^\star$. 
Then it holds that 
    \begin{align}
        \sum_{s\in\bS} d^{\pi^\star}(s) \sum_{a\notin A^\star_s}\pi(a|s) \le D_\K(\pi^\star, \pi) \le \frac{\max_{s\in\bS}\sum_{a\notin A^\star_s}\pi(a|s)}{1-\max_{s\in\bS}\sum_{a\notin A^\star_s}\pi(a|s)} ,
    \end{align}
    where $A_s\coloneqq \{ a\in\A : A^\star(s,a) = 0 \}$ denotes the set of optimal actions in $s\in\bS$. 
\end{lemma}
\begin{proof}
    We first prove the upper bound. 
To this end, we consider the state-wise information projection
\begin{align*}
    \hat\pi(\cdot|s) = \argmin \Big\{\KL(\mu, \pi(\cdot|s)) : \mu\in \Delta_\A, \operatorname{supp}(\mu)\subseteq A^\star_s \Big\} 
\end{align*}
of the policy $\pi$ to the set of optimal policies. 
By concavity we have $-\log(1-h) \le \frac{h}{1-h}$ for $h<1$ and using \Cref{lem:KLProperties} we estimate
\begin{align*}
    D_\K(\pi^\star, \pi) & \le D_\K(\hat\pi, \pi) 
    \\ & 
    = - \sum_{s\in\bS} d^{\hat\pi}(s) \KL(\hat\pi(\cdot| s), \pi_t(\cdot| s)) 
    \\ & 
    = -\sum_{s\in\bS} d^{\hat\pi}(s) \log\left(\sum_{a\in A^\star_s}\pi(a|s)\right) 
    \\ & 
    = -\sum_{s\in\bS} d^{\hat\pi}(s) \log\left(1-\sum_{a\notin A^\star_s}\pi(a|s)\right) 
    \\ & 
    \le \sum_{s\in\bS} d^{\hat\pi}(s) \cdot \frac{\sum_{a\notin A^\star_s}\pi(a|s)}{1-\sum_{a\notin A^\star_s}\pi(a|s)}. 
    \\ & 
    \le \sum_{s\in\bS} d^{\hat\pi}(s) \cdot \frac{\max_{s'\in\bS}\sum_{a\notin A^\star_{s'}}\pi(a|s')}{1-\max_{s'\in\bS}\sum_{a\notin A^\star_{s'}}\pi(a|s')}. 
    \\ & 
    = \frac{\max_{s\in\bS}\sum_{a\notin A^\star_{s}}\pi(a|s)}{1-\max_{s\in\bS}\sum_{a\notin A^\star_{s}}\pi(a|s)}. 
\end{align*}

Similar to the upper bound, we use $-\log(1-h) \ge h$ for $h<1$ and \Cref{lem:KLProperties} to estimate 
\begin{align*}
    D_\K(\pi^\star, \pi) & = \sum_{s\in\bS} d^\star(s) \KL(\pi^\star(\cdot|s), \pi(\cdot|s)) 
    \\ & 
    \ge - \sum_{s\in\bS} d^\star(s) \log\left(\sum_{a\in A^\star_s}\pi(a|s)\right)
    \\ & 
    = - \sum_{s\in\bS} d^\star(s) \log\left(1-\sum_{a\notin A^\star_s}\pi(a|s)\right)
    \\ & 
    \ge \sum_{s\in\bS} d^\star(s) \sum_{a\notin A^\star_s}\pi(a|s).
\end{align*}
\end{proof}

\begin{proof}[Proof of \Cref{thm:convergencePolicies}]
As expected, we work with \Cref{lem:boundingKakade}. 
By \Cref{lem:convergenceProbabilities} we have 
\begin{align*}
    \sum_{a\notin A^\star_s}\pi_t(a|s) \le \sum_{a\notin A^\star_s}\pi_1(a|s) e^{-\Delta(t-1) + c \log t} \le e^{-\Delta(t-1) + c \log t} 
\end{align*}
and thus 
\begin{align*}
    \frac{\max_{s\in\bS}\sum_{a\notin A^\star_s}\pi_t(a|s)}{1-\max_{s\in\bS}\sum_{a\notin A^\star_s}\pi_t(a|s)}  \le \frac{e^{-\Delta(t-1) + c \log t}}{1-e^{-\Delta(t-1) + c \log t}}
\end{align*}
if $e^{-\Delta(t-1) + c \log t}<1$ and \Cref{lem:boundingKakade} yields the upper bound. 
For the lower bound, we fix $s_0\in \bS$ and $a_0\in \A$ with $A^\star(s,a)=-\Delta$ and estimate 
\begin{align*}
    \sum_{s\in\bS} d^\star(s) \sum_{a\notin A^\star_s}\pi_t(a|s) 
    & \ge 
    \sum_{s\in\bS} d^\star(s) \sum_{a\notin A^\star_s}\pi_0(a|a) e^{A^\star(s,a)(t-1)- \gamma c \log t - 2 \lVert r \rVert_\infty}
    \\ & 
    \ge 
    d^\star(s_0) \pi_0(a_0|s_0) e^{-\Delta(t-1)- \gamma c \log t - 2 \lVert r \rVert_\infty}
    \\ & 
    \ge \min_{s\in \bS, a\notin A_s^\star} d^\star(s)\pi_0(a|s) \cdot e^{-\Delta(t-1)- \gamma c \log t - 2 \lVert r \rVert_\infty}. 
\end{align*}
\end{proof}

\subsection{Convergence of tabular unregularized natural policy gradients}
Our continuous time analysis in \Cref{lem:convergenceProbabilities} is inspired by the proof of exponential convergence of unregularized natural policy gradient methods given in~\cite{khodadadian2022linear}. 
Here, we provide an essentially sharp analysis of unregularized natural policy gradient methods in discrete time that holds for all iterations. 

Unregularized natural policy gradient ascent for a tabular softmax policy parametrization and with step size $\eta>0$ yields the explicit update rule 
\begin{align}\label{eq:NPG_update}
    \pi_{k+1}(a|s) = \frac{\pi_{k}(a|s) \exp\left(\frac{\eta A^{\pi_k}(s,a)}{1-\gamma}\right)}{Z_k(s)}, 
\end{align}
where $Z_k(s) = \sum_{a\in\A} \pi_k(a|s)e^{\eta(1-\gamma)^{-1} A^{\pi_k}(s,a)}$, see~\cite{agarwal2021theory}. 
We now characterize the optimal exponential convergence rate for natural policy gradient methods in the following setting. 

\begin{setting}\label{set:NPG}
Consider a finite discounted Markov decision process $(\bS, \A, P, \gamma, r)$, an initial distribution $\mu\in\Delta_\bS$ and fix a policy $\pi_0\in\operatorname{int}(\Delta_\A^\bS)$ and let \Cref{asu:exploration} hold and denote the optimal reward by $R^\star\coloneqq\max\{R(\pi):\pi\in\Delta_\A^\bS\}$. 
We consider iterates $(\pi_k)_{k\in\mathbb N}$ from a tabular softmax NPG without entropy regularization and a fixed stepsize $\eta>0$. 
We set 
\begin{align}
    \Delta \coloneqq -(1-\gamma)^{-1}\max\left\{ A^\star(s,a) : A^\star(s,a)\ne0, s\in\bS, a\in\A \right\}. 
\end{align}
as well as 
\begin{align}
    c \coloneqq \frac{D_\K(\pi^\star, \pi_0) + 2(1-\gamma)^{-1}\eta \lVert r \rVert_\infty}{1-\gamma}. 
\end{align}
\end{setting}

Our main result regarding the convergence of natural policy gradient methods is the following. 

\begin{theorem}[Convergence of unregularized natural policy gradients]\label{thm:convregenceNPG}
Consider \Cref{set:NPG}. 
For all $k\in\mathbb N_{>0}$ it holds that 
\begin{align}
    R^\star-R(\pi_k) & \le \frac{2\lVert r \rVert_\infty}{1-\gamma} \cdot e^{-\Delta \eta (k-1) + c \log k+c} \quad \text{as well as} \\ 
    R^\star-R(\pi_k) & \ge \Delta \cdot \min\limits_{s\in\bS} d^{\pi_k}(s)\sum_{a\notin A_s^\star}\pi_0(a|s) \cdot e^{-\Delta\eta (k-1) -  c \log k - c-  \frac{4\eta \lVert r \rVert_\infty}{1-\gamma}}. 
\end{align}
\end{theorem} 

Similar upper and lower bounds in an asymptotic regime were given in~\cite{khodadadian2022linear,liu2024elementary}, respectively, where our guarantees hold for all iterates. 
The proof of \Cref{thm:convregenceNPG} is analogous to the analysis in continuous time, and is built on the following the discrete-time analog of \Cref{cor:convValueFunction}. 

\begin{restatable}[Sublinear convergence of NPG]{proposition}{sublinear}
\label{prop:sublinearNPG}
Consider \Cref{set:NPG}. 
Then for all $k\in\mathbb N_{>0}$ it holds that 
    \begin{align}\label{eq:sublinearNPG}
        0\le V^\star(s) - V^{\pi_k}(s) \le \frac{D_\K(\pi^\star, \pi_0) + 2(1-\gamma)^{-1}\eta \lVert r \rVert_\infty}{\eta k} \quad \text{for all } s\in\bS.
    \end{align}
    In particular, for all $s\in\bS$ and $a\in\A$ it holds that 
    \begin{align}\label{eq:sublinearAdvantageNPG}
        A^{\pi_k}(s,a) & \le A^\star(s,a) +\frac{D_\K(\pi^\star, \pi_0) + 2(1-\gamma)^{-1} \eta \lVert r \rVert_\infty}{\eta k} \quad \text{and } \\ 
        A^{\pi_k}(s,a) & \ge A^\star(s,a) -\gamma \cdot \frac{D_\K(\pi^\star, \pi_0) + 2(1-\gamma)^{-1} \eta \lVert r \rVert_\infty}{\eta k}.
    \end{align}
\end{restatable}

\Cref{prop:sublinearNPG} is an extension of the convergence analysis in~\cite[Theorem 16]{agarwal2021theory} and we defer its proof to \Cref{app:sec:convergence}. 
Similar to \Cref{lem:convergenceProbabilities}, we provide exponential bounds on the entries $\pi_k(a|s)$ of the policies. 
Our analysis is similar to the one in~\cite[Lemma 3.4]{khodadadian2022linear}, however, our bound holds for arbitrary initial policies, for all iterations, and comes with a lower bound matching up to a polynomial term. 

\begin{lemma}\label{lem:convergenceProbabilitiesNPG}
Consider \Cref{set:NPG}. 
Then for $k\in\mathbb N_{>0}$ it holds that 
\begin{align}
    \begin{split}
        \pi_k(a|s) & \le \pi_1(a|s) \exp\left(\frac{A^\star(s,a)\eta (k-1)}{1-\gamma} +  c \log k + c\right) \quad \text{as well as} \\ 
        \pi_k(a|s) & \ge \pi_1(a|s) \exp\left(\frac{A^\star(s,a)\eta (k-1)}{1-\gamma} - c \log k\right),
    \end{split}
\end{align}
and 
\begin{align}
    \begin{split}
        \pi_k(a|s) & \le \pi_0(a|s) \exp\left(\frac{A^\star(s,a)\eta (k-1)}{1-\gamma} +  c \log k + c +\frac{2\eta \lVert r \rVert_\infty}{1-\gamma}\right) \quad \text{as well as} \\ 
        \pi_k(a|s) & \ge \pi_0(a|s) \exp\left(\frac{A^\star(s,a)\eta k}{1-\gamma} - c \log k - c - \frac{4\eta \lVert r \rVert_\infty}{1-\gamma}\right).
    \end{split}
\end{align}
\end{lemma}
\begin{proof}
By \Cref{lem:ascentNPG} we have $\log Z_k(s)\ge 0$ which implies $Z_k(s)\ge1$ and
\begin{align*}
    \pi_{k+1}(a|s) \le \pi_{k}(a|s)\exp\left(\frac{\eta A^{\pi_k}(s,a) }{1-\gamma}\right). 
\end{align*}
Iterating over $k$ and using \Cref{eq:sublinearAdvantageNPG} as well as $\lVert A^\pi \rVert_\infty \le 2 \rVert r \rVert_\infty$ we obtain 
\begin{align*}
    \pi_K(a|s) & \le \pi_1(a|s) \exp\left(\frac{\eta}{1-\gamma} \sum_{k=1}^{K-1} A^{\pi_k}(s,a) \right) 
    \\ & 
    \le \pi_1(a|s) \exp\left(\frac{\eta}{1-\gamma} \sum_{k=1}^{K-1} A^{\pi_k}(s,a) \right)
    \\ & 
    \le \pi_1(a|s) \exp\left( \frac{A^\star(s,a) \eta(K-1)}{1-\gamma} +  c \sum_{k=1}^{K-1} k^{-1} \right) 
    \\ & 
    \le \pi_1(a|s) \exp\left(\frac{A^\star(s,a)\eta(K-1)}{1-\gamma} +  c(1+\log K) \right) ,
\end{align*}
where we used 
\begin{align*}
    \sum_{k=1}^{K-1} k^{-1} = 1 + \sum_{k=2}^{K-1} k^{-1} \le 1 + \int_{1}^{K} t^{-1} \mathrm{d}t = 1 + \log K. 
\end{align*}
Further, it holds that 
\begin{align*}
    \pi_{1}(a|s) \le \pi_{0}(a|s)\exp\left(\frac{\eta A^{\pi_0}(s,a) }{1-\gamma}\right) \le \pi_{0}(a|s)\exp\left(\frac{2\eta \lVert r \rVert_\infty}{1-\gamma}\right). 
\end{align*}

The argument for the lower bound works similarly, however, we use \Cref{lem:ascentNPG} as this yields 
\begin{align*}
    \pi_{k+1}(a|s) & = \pi_k(a|s) \exp\left(\frac{\eta A^{\pi_k}(s,a)}{1-\gamma} - \log Z_k(s) \right) 
    \\ & 
    \ge \pi_k(a|s) \exp\left(\frac{\eta A^{\pi_k}(s,a)}{1-\gamma} - \frac{\eta(V^{\pi_{k+1}}(s) - V^{\pi_{k}}(s))}{1-\gamma} \right) 
    \\ & 
    \ge \pi_k(a|s) \exp\left(\frac{\eta A^{\pi_k}(s,a)}{1-\gamma} - \frac{\eta (V^{\star}(s) - V^{\pi_{k}}(s))}{1-\gamma} \right) 
    \\ & 
    \ge \pi_k(a|s) \exp\left(\frac{\eta A^{\star}(s,a)}{1-\gamma} - c k^{-1} \right),
\end{align*}
where we used $V^{\pi_{k+1}}(s)\le V^{\star}(s)$, see \Cref{lem:ascentNPG}. 
Iterating over $k$ yields 
\begin{align*}
    \pi_K(a|s) & \ge \pi_1(a|s) \exp\left(
    \frac{\eta A^{\star}(s,a) (K-1)}{1-\gamma}
    - c \sum_{k=1}^{K-1} k^{-1} \right)
    \\ & \ge \pi_1(a|s) \exp\left(
    \frac{\eta A^{\star}(s,a) (K-1)}{1-\gamma}
    - c \log K \right), 
\end{align*}
where we  used 
\begin{align*}
    \sum_{k=1}^{K-1} k^{-1} \ge \int_{1}^{K} t^{-1} \mathrm{d}t = \log K. 
\end{align*}
Finally, we have 
\begin{align*}
    \pi_{1}(a|s) \ge \pi_{0}(a|s)\exp\left(\frac{\eta A^{\pi_0}(s,a)}{1-\gamma} - \frac{\eta (V^{\star}(s) - V^{\pi_{0}}(s))}{1-\gamma} \right) \ge  
    \exp\left(-\frac{4\eta \lVert r \rVert_\infty}{1-\gamma} \right). 
\end{align*}
\end{proof}

\begin{proof}[Proof of \Cref{thm:convregenceNPG}]
The statements follow just like in the proof of \Cref{thm:convergenceValue}, where here we combine \Cref{lem:subOptimalityGap} and \Cref{lem:boundingKakade} with \Cref{lem:convergenceProbabilitiesNPG}. 
\end{proof}

\section{The Case of General Convex Regularizers} 
Although entropy regularization is arguably the most classic regularization strategy in reinforcement learning, more general regularizers have been considered in a variety of works~\cite{neu2017unified, geist2019theory, lan2022policy, zhan2021policy}. 
It is the purpose of this section to extend our explicit analysis of the entropy regularization to general convex regularizers. 
We show that for general convex regularizers considered in reinforcement learning, a corresponding linear programming formulation with convex regularization can be derived. 
In particular, this allows us to extend the implicit bias results to this more general class of policy mirror descent methods. 
Finally, we discuss for which regularizers one obtains an explicit expression of the policy gradient, allowing for an explicit convergence analysis. 
For ease of notation, we use the shorthand notation $\pi_s = \pi(\cdot|s)$ for a policy $\pi\in\Delta_\A^\bS$ as well as $\nu_s = \nu(s,\cdot)$ for a vector $\mathbb R^{\bS\times\A}$. 

Here, we consider a general convex regularizer $\psi\colon\Delta_{\A}\to\mathbb R$ and consider the regularized reward 
\begin{align}
    R_\tau(\pi) \coloneqq (1-\gamma) \E_{\PP^{\pi, \mu}}\left[ \sum_{t\in\N} \gamma^t (r(S_t, A_t) - \tau \psi(\pi_{S_t})) \right]. 
\end{align}
Just like for the entropy, the regularizer term can be expressed as $\Psi(\pi) = \sum_{s\in\bS} d^{\pi}(s) \psi(\pi(\cdot|s))$. 
Note that although we can, without loss of generality, assume $\Psi(\mu) = D_\Psi(\pi,\pi_0)$. 
Otherwise, we simply choose $\pi_0(\cdot|s)$ to be the maximizer of $\psi$. 

We obtain the following expression as a regularized linear program 
\begin{align}\label{eq:LP-general}
    \max r^\top \nu - \tau \Phi(\nu) \quad \text{subject to } \nu\in\cD, 
\end{align}
where the regularizer is given by 
\begin{align}\label{eq:SA-regularizer}
    \Phi(\nu) = \Psi(\nu(\cdot|\cdot)) = 
    \sum_{s\in\bS} \left(\sum_{a\in\A} \nu(s,a)\right)\psi\left(\frac{\nu(s,\cdot)}{\sum_{a\in\A} \nu(s,a)}\right). 
\end{align}
The direct generalization of the Kakade metric is given by 
\begin{align}\label{eq:general-policy-metric}
    g^\psi_\pi(u, w) \coloneqq \sum_{s\in\bS} d^\pi(s) u_s^\top \nabla^2\psi(\pi_s) w_s,
\end{align}
where $u_s = u(s, \cdot)$ and $w_s = w(s,\cdot)$ and $\pi_s = \pi(\cdot|s)$. 
We fix the setting we work in for this section. 

\begin{setting}\label{set:general}
Consider a finite discounted Markov decision process $(\bS, \A, P, \gamma, r)$, an initial distribution $\mu\in\Delta_\bS$ and fix a policy $\pi_0\in\operatorname{int}(\Delta_\A^\bS)$, assume that \Cref{asu:exploration} holds and denote the optimal reward by $R^\star\coloneqq\max\{R(\pi):\pi\in\Delta_\A^\bS\}$. 
Further, we fix a strictly convex regularizer $\psi\colon\Delta_\A\to\mathbb R$ that is twice continuously differentiable on $\operatorname{int}(\Delta_\A)$ with 
$v^\top\nabla^2\psi(\mu) v> 0$ for all $v\in T\Delta_\A$ and $\mu\in\operatorname{int}(\Delta_\A)$. 
\end{setting}

Under the assumptions above,~\eqref{eq:general-policy-metric} indeed defines a metric, which -- just like the Kakade metric -- is isometric to the Hessian metric induced by the regularizer $\Phi$ is a generalization of \Cref{thm:SAGeometryKakade}.

\begin{theorem}\label{thm:isometry-general}
    The regularizer $\Phi$ given by $\Phi(\nu) = \Psi(\nu(\cdot|\cdot))$ is convex. 
    Let further \Cref{set:general} hold. Then the mapping
    \begin{align}
        \begin{split}
            \Xi\colon(\operatorname{int}(\Delta_\A^\bS), g^\psi) & \to (\operatorname{int}(\cD), g^\Phi) \\ 
            \pi & \mapsto \nu^\pi \\ \nu(\cdot|\cdot) & \; \leftmapsto \; \nu 
        \end{split}
    \end{align}
    is an isometric diffeomorphism, where $g^\Phi$ denotes the Hessian metric induced by $\Phi$. 
\end{theorem}

The following general observation about normalization is useful in our analysis. 

\begin{lemma}\label{lem:perspective-geometry}
 Let $\psi\colon \Delta_\A \to \mathbb{R}$ be a twice-differentiable convex function defined on the probability simplex, and consider the function  
\begin{align}\label{eq:definition-scaled-perspective}
    \phi(\mu)= \left(\sum_{a\in \A} \mu(a)\right) \psi\!\left(\frac{\mu}{\sum_{a\in \A}^n \mu(a)}\right),
\end{align}
which is defined for for $\mu\in\mathbb{R}^\A_{>0}$. 
Then $\phi$ is convex.  
Further, we set $t=\sum_{a\in \A} \mu(a)$ and $p=\frac{\mu}{t}$ and consider the normalization map $F\colon\mathbb{R}^\A_{>0}\to\Delta_\A$ given by $F(\mu)=\frac{\mu}{t}$, then 
\begin{align}\label{eq:help-perspective-geometry}
    u^\top \nabla^2 \phi(\mu) w = t \cdot  (dF_\mu u)^\top \nabla^2 \psi(p) dF_\mu w. 
\end{align}
\end{lemma}
\begin{proof}
The \emph{perspective} of $\psi$ is given by $f(\mu, t) = t \psi(\frac{\mu}{t})$, which is convex, see~\cite[Subsection 3.2.6]{boyd2004convex}. 
As $\mu\mapsto(\mu, \sum_{a\in \A}\mu(a))$ is linear, $\phi$ is convex as the composition of a linear and convex function. 

Further, direct computation shows that the Hessian of the perspective $f$ is given by 
\begin{align*}
    \nabla^2 f(\mu, t) = \begin{pmatrix}
        \frac1t \nabla^2\psi(\frac\mu t) & \frac1{t^2} \nabla^2\psi(\frac\mu t)\mu \\ \frac1{t^2} \mu^\top \nabla^2\psi(\frac\mu t) & \frac1{t^3} \mu^\top \nabla^2\psi(\frac\mu t)\mu 
    \end{pmatrix} = \begin{pmatrix}
        \frac1t \nabla^2\psi(p) & \frac1{t^2} \mu^\top \nabla^2\psi(p) \\ \frac1{t^2} \nabla^2\psi(p)\mu & \frac1{t^3} \mu^\top \nabla^2\psi(p)\mu 
    \end{pmatrix}. 
\end{align*}
We denote the mapping $g(\mu) = (\mu, \sum_{a\in\A} \mu(a))=(\mu, \mathds{1}^\top \mu)$, where $\mathds{1}$ denotes the all one vector. 
By the chain rule for Hessian matrices, we have $\nabla^2 \phi(\mu) = Dg(\mu)^\top \nabla^2 f(\mu, t) Dg(\mu)$. 
Note that $g$ is linear and hence, we have $Dg(\mu)v = g(v)$ and we compute 
\begin{align}
    v^\top \nabla^2 \phi(\mu) w = g(v)\top \nabla^2 f(\mu, t) g(w) = \frac1t \left( v - \frac{\mathds{1}^\top v}t \mu \right)^\top \nabla^2 \psi(p) \left( v - \frac{\mathds{1}^\top v}t \mu \right). 
\end{align}
Finally, direct computation gives 
\[
    dF_\mu v = \frac{1}{t}v - \frac{\mathds{1}^\top v}{t^2}\mu,
\]
which concludes the proof
\end{proof}

\begin{proof}[Proof of \Cref{thm:isometry-general}]
For the convexity of $\Phi(\nu) = \sum_{s\in\bS} \phi(\nu_s)$ we can apply \Cref{lem:perspective-geometry} to $\phi$. 
    
Under \Cref{asu:exploration}, the mapping $\Xi$ is a smooth bijection with smooth inverse, hence a diffeomorphism. 
    It remains to show the isometry property
    \begin{align}
        g^\Phi_\nu(u, w) = g^\psi_\pi(d\Xi^{-1}_\nu u, d\Xi^{-1}_\nu w). 
    \end{align}
    To validate this, we recall that $\Xi^{-1}$ is given by conditioning, which corresponds to normalizing each block $\nu_s = \nu(s, \cdot)$. 
    Hence, we can use~\eqref{eq:help-perspective-geometry} to compute 
    \begin{align*}
        g^\psi_\pi(d\Xi^{-1}_\nu u, d\Xi^{-1}_\nu w) & = \sum_{s\in\bS} d^\pi(s) ((d\Xi^{-1}_\nu u)_s)^\top \nabla^2\psi(\pi_s) ((d\Xi^{-1}_\nu w)_s) 
        \\ & = \sum_{s\in\bS} \left(\sum_{a\in\A} \nu_s(a)\right)  (dF_{\nu_s} u_s)^\top \nabla^2\psi(\pi_s) (dF_{\nu_s} w_s )
        \\ & = \sum_{s\in\bS} u_s^\top \nabla^2\phi(\nu_s) w_s 
        \\ & = u^\top \nabla^2\Phi(\nu) w . 
    \end{align*}
\end{proof}

For general regularized MDPs, \Cref{thm:isometry-general} guarantees that we obtain a linear programming formulation~\eqref{eq:LP-general} with a convex regularizer $\Phi$. 
Consequently, we can apply the theory of Hessian gradient flows, which implies that the solutions $\pi^\star_t$ of the regularized problems (if unique) follow a gradient flow principle. 
We denote the gradient of $R$ with respect to the Riemannian metric $g^\psi$ by $\nabla^\psi R$ and say that $(\pi_t)_{t\in[0,T)}\subseteq \operatorname{int}(\Delta_\A^\bS)$ solves the \emph{$\psi$ gradient flow} if 
\begin{align}\label{eq:psi-GF}
    \partial_t \pi_t = \nabla^\psi R(\pi_t).
\end{align}
Further, the direct extension of the Kakade divergence is given by 
\begin{align}
    D_\Psi(\pi_1, \pi_2) = \sum_{s\in\bS} d^{\pi_1}(s) D_\psi(\pi_1(\cdot|s), \pi_2(\cdot|s)),
\end{align}
where $D_\psi$ denotes the Bregman divergence induced by $\psi$. 
Now we can formulate the analog to \Cref{cor:implications} regarding the central path property, sublinear convergence, and implicit bias. 

\begin{theorem}
Consider a finite Markov decision process, let \Cref{asu:exploration} hold, and consider a strictly convex regularizer $\psi\colon\Delta_\A\to\mathbb R$ with positive definite Hessian. 
Let \Cref{asu:exploration} hold, denote the set of optimal policies by  $\Pi^\star\coloneqq\{ \pi\in\Delta_\A^\bS : R(\pi) = R^\star \}$, and fix an initial policy $\pi_0\in\operatorname{int}(\Delta_\A^\bS)$. 
Then the following statements hold: 
\begin{enumerate}
    \item \emph{Well-posedness:} 
    The gradient flow~\eqref{eq:psi-GF} admits a unique local solution. 
    \item \emph{Central path property:} It holds that 
    \begin{align}
        \pi_t = \argmax_{\pi\in\Delta_\A^\bS} \Big\{R(\pi) - t^{-1} D_\Psi(\pi,\pi_0)\Big\}. 
    \end{align}
    \item \emph{Sublinear convergence:} We have 
    \begin{align}
            0\le R^\star - R(\pi_t)\le \frac{\min_{\pi^\star\in\Pi^\star}D_\Psi(\pi^\star, \pi_0) - D_\Psi(\pi_t, \pi_0)}{t} \le \frac{\min_{\pi^\star\in\Pi^\star}D_\Psi(\pi^\star, \pi_0)}{t}.
    \end{align}
    \item \emph{Implicit bias:} Under global well-posedness, the gradient flow $(\pi_t)_{t\ge0}$ converges globally and it holds that 
    \begin{align}
        \lim_{t\to+\infty} \pi_t = \pi^\star = \argmin_{\pi\in\Pi^\star} D_\Psi(\pi, \pi_0). 
    \end{align}    
\end{enumerate}
\end{theorem}
\begin{proof}
The proof is analogous to~\Cref{cor:implications}. 
\end{proof}

The global well-posedness of the gradient flows can be checked for specific choices of regularizers $\psi$ by verifying the Legendre-type property $\lVert \nabla\Phi(\nu) \rVert\to+\infty$ for $\nu\to\partial\cD$. 

In particular, this describes the regularization path as the solution of a gradient flow or, equivalently, characterizes the implicit bias of the gradient flow. 

\begin{theorem}
    Let \Cref{asu:exploration} hold. 
Then for all $\pi\in\operatorname{int}(\Delta_\A^\bS)$ and $a\in\A$, $s\in\bS$ it holds that 
\begin{equation}
\nabla^\psi R(\pi)(s,a) = (1-\gamma)^{-1} u(s,a),
\end{equation}
where $u$ is uniquely defined by $u_s = \nabla^2\psi(\pi_s)^{-1} A^\pi_s$. 
\end{theorem}
\begin{proof}
Just like in the case of~\Cref{thm:PGtheorem}, we obtain that 
\begin{align*}
    \partial_w R(\pi) = (1-\gamma)^{-1}\sum_{s,a} d^\pi(s) A^\pi(s,a) w(s,a) = (1-\gamma)^{-1}g^\psi_\pi(u,w), 
\end{align*}
where $u$ satisfies $u_s = \nabla^2\psi(\pi_s)^{-1} A^\pi_s$.
\end{proof}

In particular, for regularizers of the form 
\begin{align}
    \psi(\mu) = \sum_{a\in\A} h(\mu(a))
\end{align}
for a suitably convex and regular function $h$, the Riemannian metric is given by 
\begin{align}
        g^\psi_\pi(u, w) \coloneqq \sum_{s\in\bS} d^\pi(s) u(s,a) h''(\pi(a|s)) w(s,a),
\end{align}
and the policy gradient theorem takes the form 
\begin{align}
\nabla^\K R(\pi)(s,a) = (1-\gamma)^{-1} \cdot \frac{A^\pi(s,a)}{h''(\pi(a|s))}. 
\end{align} 

If further $h''(x) = x^{-\sigma}$ for $\sigma\in(1,+\infty)$, we can proceed analogously to the case of entropy regularization. 
Here, we obtain, together with the sublinear convergence, the differential inequality 
\begin{equation*}
    \left(A^{\star}(s,a) - \frac{\gamma D_\Psi(\pi^\star, \pi_0)}t\right)\pi_t(a|s)^\sigma \le (1-\gamma) \partial_t \pi_t(a|s) \le \left(A^{\star}(s,a) + \frac{ D_\Psi(\pi^\star, \pi_0)}{t}\right)\pi_t(a|s)^\sigma.
\end{equation*}
Note that $\partial_t x(t) = \alpha(t) x(t)^{\sigma}$ is solved by 
\begin{align*}
    x(t) = \left((1-\sigma) \int_0^t \alpha(r) \mathrm dr + x(0)^{1-\sigma}\right)^{-\frac1{\sigma-1}}.
\end{align*}
Hence, we obtain for all $t\ge1$ that 
\begin{align}\label{eq:upper-bound-sigma}
    \pi_t(a|s) \le \left( \frac{(1-\sigma) ((t-1) A^\star(s,a) + D_\Psi(\pi^\star, \pi_0) \log(t) + 2\lVert r \rVert_\infty)}{1-\gamma} + \pi_{0}(a|s)^{1-\sigma}\right)^{-\frac1{\sigma-1}}
\end{align}
as well as 
\begin{align}\label{eq:lower-bound-sigma}
    \pi_t(a|s) \ge \left( \frac{(1-\sigma) ((t-1) A^\star(s,a) + D_\Psi(\pi^\star, \pi_0) \log(t) - 2\lVert r \rVert_\infty)}{1-\gamma} - \pi_{0}(a|s)^{1-\sigma}\right)^{-\frac1{\sigma-1}}. 
\end{align}
Now we can formulate the optimal convergence rates for a class of convex regularizers. 

\begin{theorem}[Convergence rates] 
    Consider \Cref{set:general}, fix $\sigma\in (1,\infty)$ and assume that the convex regularizer is given by 
    \begin{align}
        \psi(\mu) = \begin{cases}
            \sum_{a\in\A} \log \mu(a) \quad & \text{if } \sigma=2 \\
            \frac1{(2-\sigma)(1-\sigma)}\sum_{a\in\A} \mu(a)^{2-\sigma} \quad & \text{otherwise}. 
        \end{cases}
    \end{align}
    Further, 
    \begin{align}
        \pi_t = \argmax_{\pi\in\Delta_\A^\bS} \Big\{R(\pi) - t^{-1} D_\Psi(\pi,\pi_0)\Big\}. 
    \end{align}
    We denote the generalized maximum entropy optimal policy by $\pi^\star = \argmin_{\pi\in\Pi^\star} D_\Psi(\pi, \pi_0)$, where $\Pi^\star\coloneqq\{ \pi\in\Delta_\A^\bS : R(\pi) = R^\star \}$ denotes the set of optimal policies.   
    Then, there exist constants $c_1, c_2, C>0$ depending on $\sigma$, $A^\star$, $\pi_0$, $\gamma$, $\lVert r \rVert_\infty$ and $\min_{s\in\bS} d^\star(s)$ such that 
    \begin{align}
        c_1 (t+\log(t)-C)^{-\frac{1}{\sigma-1}} \le R^\star - R(\pi_t) \le c_2 (t-\log(t)-C)^{-\frac{1}{\sigma-1}}
    \end{align}
    as well as 
    \begin{align}
        c_1 (t+\log(t)-C)^{-\frac{1}{\sigma-1}}  \le D_\K(\pi^\star, \pi_t) \le c_2 (t-\log(t)-C)^{-\frac{1}{\sigma-1}}. 
    \end{align}
\end{theorem}
\begin{proof}
First, we show the global well-posedness of the gradient flows in this case. 
Here, note that it suffices to show the Legendre-type property $\lVert \nabla\Phi(\nu_k) \rVert \to +\infty$ for $\nu_k\to\nu\to\partial\cD$, see~\cite{alvarez2004hessian, mueller2023thesis}. 
Fix $\nu_k\to\nu\partial\cD$ and denote the corresponding state distributions by $d_k\to d$. 
Then we have $\nu_k(s,a)\to0$ for some $s\in\bS,a\in\A$ and by~\Cref{asu:exploration} we have $d(s)>0$. 
Hence, we have $\pi_k(a|s)\to0$. 
Recall that we can express the regularizer as $\Phi(\nu) = \sum_{s\in\bS}\phi(\nu_k(s,\cdot))$ with $\phi$ defined in \eqref{eq:definition-scaled-perspective} and hence it remains to show that $\lVert \nabla \phi(\eta(s, \cdot)) \rVert\to+\infty$. 
Direct computation yields 
\begin{align*}
    \nabla \phi(\eta(s, \cdot)) = \psi(\pi_s) \mathds{1} + \nabla \psi(\pi_s) - (\nabla \psi(p)^\top p)\mathds{1}. 
\end{align*}
Examining the entry corresponding to $a$ we obtain
\begin{align*}
    \frac{\pi(a|s)^{2-\sigma}}{(2-\sigma)(1-\sigma)}+\frac{\pi(a|s)^{1-\sigma}}{1-\sigma}-\frac{\pi(a|s)^{2-\sigma}}{1-\sigma} \to +\infty
\end{align*}
for $\pi(a|s)\to0$, which establishes the Legendre-type property. 

Now, the \Cref{lem:subOptimalityGap,lem:boundingKakade} in combination with~\eqref{eq:upper-bound-sigma} and \eqref{eq:lower-bound-sigma} yield the result. 
\end{proof}

Note that 
\begin{align*}
    \lim_{t\to+\infty} \frac{c_1 (t+\log(t)-C)^{-\frac{1}{1-\sigma}}}{c_2 (t-\log(t)-C)^{-\frac{1}{1-\sigma}}} = \frac{c_1}{c_2}
\end{align*}
showing that the bound is asymptotically optimal up to absolute constants. 

\section{Improved Iteration Complexity for Regularized Natural Policy Gradients}\label{sec:overall}

Entropy regularization is commonly added to encourage exploration and accelerate the optimization process, however, the unregularized objective is still the objective criterion one wishes to optimize. 
Hence, it is a natural question to ask what accuracy one can achieve with a budget of $k$ iterations with a method that aims to optimize the regularized reward. 
One way to approach this is to use the error decomposition
\begin{align}\label{eq:errorDecomposition}
    0\le R^\star - R(\pi_k) = R^\star - R(\pi^\star_\tau) + R(\pi^\star_\tau) - R_\tau(\pi_\tau^\star) + R_\tau(\pi^\star_\tau) - R_\tau(\pi_k)  + R_\tau(\pi_k) - R(\pi_k),
\end{align} 
Applying the $R(\pi) - R_\tau(\pi) = O(\tau)$ bound on the entropy regularization error gives 
\begin{align}
    0\le R^\star - R(\pi_k) = O(\tau) + R_\tau(\pi^\star_\tau) - R_\tau(\pi_k) = O(\tau + e^{-\tau \eta k}) = O\left(\frac{\log k}{\eta k}\right)
\end{align}
for entropy-regularized natural policy gradients with stepsize $\eta = \frac{\log k}{k}>0$, see~\cite{cen2021fast}. 
See also~\cite{sethi2024entropy} for an $O(\frac{1}{t})$ guarantee of entropy-regularized natural policy gradient flows with $\tau = \frac1t$. 
In contrast, unregularized natural policy gradient achieves an exponential convergence rate of $\tilde{O}(e^{-\Delta \eta k})$, see \Cref{thm:convregenceNPG} and~\cite{khodadadian2022linear, liu2024elementary}. 
We use our estimate on the regularization error and obtain the following improved guarantee for entropy-regularized natural policy gradients. 

\begin{theorem}[Overall error analysis]\label{thm:overallError}
Consider a regularization strength $\tau\in(0,1]$ and consider the entropy-regularized reward $R_\tau(\pi) = R(\pi) - \tau D_\K(\pi, \pi_{\textup{unif}})$, $\pi_{\textup{unif}}$ denotes the uniform policy, meaning $\pi_{\textup{unif}}(a|s)= \lvert A\rvert^{-1}$ for all $a\in\A$, $s\in\bS$. 
We denote the optimal entropy regularized policy by $\pi^\star_\tau = \argmax_{\pi\in\Delta_\A^\bS} R_\tau(\pi)$ and the maximum entropy optimal policy by $\pi^\star$, meaning $\pi^\star(a|s) = \lvert A_s^\star \rvert^{-1}$ if $a\in A^\star_s$ is an optimal action. 
Assume that $(\pi_k)_{k\in\mathbb N}$ be the iterates produced by natural policy ascent with a log-linear tabular policy parametrization with stepsize $0<\eta\le\frac{1-\gamma}{\tau}$. 
Then it holds that 
\begin{align}\label{eq:overallErrorBound}
    R^\star - R(\pi_{k+1}) \le  
    \frac{2\lVert r \rVert_\infty e^\Delta}{1-\gamma} \cdot \tau^{-c} e^{-\Delta \tau^{-1}} 
    + 
    \frac{2\cdot \lvert \bS \rvert \cdot \lVert r \rVert_\infty C^{\frac12}}{1-\gamma} \cdot \tau^{-\frac12} e^{-\frac{\eta\tau (k-1)}{2}}, 
\end{align}
where $c=D_\K(\pi^\star, \pi_{\textup{unif}}) = \sum_{s\in\bS} d^{\pi^\star}(s)\log\frac{\lvert \A \rvert}{\lvert A^\star_s \rvert}\le\log\lvert \A\rvert$ and 
\begin{align} \label{eq:definitionConstantTheoremCen}
        C = \lVert Q_\tau^{\pi^\star_\tau} - Q_\tau^{\pi^{(0)}} \rVert_\infty + 2\tau \left( 1 - \frac{\eta\tau}{1-\gamma} \right) \lVert \log \pi^\star_\tau - \log \pi_0 \rVert_\infty. 
    \end{align}  
In particular, choosing $\tau = \sqrt{\frac{2\Delta}{\eta k}}$ yields 
\begin{align}\label{eq:overallRate}
    R^\star - R(\pi_k) = O\left(\left((\eta k)^{\frac{c}2} + (\eta k)^{\frac{1}4}\right)\cdot e^{-\sqrt{\frac{\Delta \eta k}{2}} } \right). 
\end{align}
\end{theorem}

In practice, one does not have access to the problem-dependent constant $\Delta>0$ without solving the reward optimization problem. 
However, setting $\tau = (\alpha \eta k)^{-\frac12}$ for some $\alpha>0$ yields an overall error estimate of 
\begin{align}
    R^\star - R(\pi_k) = O\left(\left(\frac{\eta k}{\alpha}\right)^{\frac{c}2} e^{-\Delta \sqrt{{\alpha\eta k}}} + \left(\frac{\eta k}{\alpha}\right)^{\frac{1}4} e^{-\sqrt{{\alpha\eta k}}} \right). 
\end{align}

In our proof, we use the following stability result on the reward function. 

\begin{proposition}[Lipschitz-continuity of the reward~\cite{mueller2023thesis}]\label{prop:Lipschitz}
    It holds that 
    \begin{align}
        \lvert R(\pi_1) - R(\pi_2) \rvert \le \frac{\lVert r \rVert_\infty}{1-\gamma} \cdot \lVert \pi_1 - \pi_2 \rVert_1 \quad \text{for all } \pi_1, \pi_2 \in\Delta_\A^\bS. 
    \end{align}
\end{proposition}

Now we can upper bound the suboptimality gap $R^\star- R(\pi)$ of a policy in terms of $R^\star- R(\pi^\star_\tau)$ and the distance between $\pi$ and $\pi^\star_\tau$. 

\begin{lemma}\label{prop:errorDecomposition}
For any policy $\pi\in\Delta_\A^\bS$ we have 
\begin{align}
    0\le R^\star - R(\pi) \le R^\star - R(\pi^\star_\tau) + \frac{\sqrt{2} \cdot \lvert \bS \rvert \cdot \lVert r \rVert_\infty}{1-\gamma} \cdot \lVert \log \pi^\star_\tau - \log \pi \rVert_\infty^{\frac12} . 
\end{align}
\end{lemma}
\begin{proof}
First, note that $R^\star-R(\pi) = R^\star-R(\pi^\star_\tau) + R(\pi^\star_\tau)-R(\pi)$. 
Using the Lipschitz continuity of the reward from \Cref{prop:Lipschitz} as well as Pinsker's and Jensen's inequality, we estimate
\begin{align*}
    R(\pi^\star_\tau)-R(\pi) & \le \frac{\rVert r \rVert_\infty}{1-\gamma} \cdot \lVert \pi^\star_\tau - \pi \rVert_1 
    \\ & = 
    \frac{\rVert r \rVert_\infty}{1-\gamma} \cdot \sum_{s\in\bS} \lVert \pi^\star_\tau(\cdot|s) - \pi(\cdot|s) \rVert_1
    \\ & \le 
    \frac{\sqrt{2} \cdot \rVert r \rVert_\infty}{1-\gamma} \cdot \sum_{s\in\bS} \KL(\pi^\star_\tau(\cdot|s), \pi(\cdot|s))^{\frac12}
    \\ & = 
    \frac{\sqrt{2} \cdot \rVert r \rVert_\infty}{1-\gamma} \cdot \sum_{s\in\bS} \left(\sum_{a\in\A} \pi^\star_\tau(a|s) \log \frac{\pi^\star_\tau(a|s)}{\pi(a|s)} \right)^{\frac12}
    \\ & \le 
    \frac{\sqrt{2} \cdot \rVert r \rVert_\infty}{1-\gamma} \cdot \sum_{s\in\bS} \left(\sum_{a\in\A} \pi^\star_\tau(a|s) \lVert \log \pi^\star_\tau - \log \pi \rVert_\infty \right)^{\frac12}
    \\ & =
    \frac{\sqrt{2} \cdot \lvert \bS \rvert \cdot \lVert r \rVert_\infty}{1-\gamma} \cdot \lVert \log \pi^\star_\tau - \log \pi \rVert_\infty^{\frac12}. 
\end{align*}
\end{proof}

\Cref{prop:errorDecomposition} can by used in combination with any result bounding $D_\K(\pi^\star_\tau, \pi_k)$ for a policy optimization technique. 
If only bounds on $R_\tau(\pi^\star_\tau) - R_\tau(\pi_k)$ are available, one can also use the local bound $D_\K(\pi^\star_\tau, \pi) \le \omega \tau^{-1} (R_\tau(\pi^\star_\tau) - R_\tau(\pi))$, which holds in a neighborhood of $\pi^\star_\tau$ that depends on $\omega\in(0,1)$, see~\cite[Lemma 29]{muller2024geometry}. 
Here, we limit our discussion to entropy-regularized natural policy gradient methods, which are known to converge linearly. 

\begin{theorem}[Convergence of entropy-regularized NPG, \cite{cen2021fast}]\label{thm:convergenceRegularizedNPG}
Consider natural policy gradient with a tabular softmax policy parametrization, a fixed regularization strength $\tau>0$, and stepsize $0<\eta\le\frac{1-\gamma}{\tau},$ and denote the iterates of the natural policy gradient updates by $(\pi_k)_{k\in\mathbb N}$.
Then for $k\in\mathbb N$, it holds that 
    \begin{align}
        \lVert Q^{\pi^\star_\tau}_\tau - Q_\tau^{\pi_{k+1}} \rVert_\infty & \le C (1-\eta\tau)^k \quad \text{and } \\
        \lVert \log \pi^\star_\tau - \log \pi_{k+1} \rVert_\infty & \le 2C \tau^{-1} (1-\eta\tau)^k, 
    \end{align} 
    where $C$ is defined in~\eqref{eq:definitionConstantTheoremCen}
\end{theorem}

Now we can provide the proof of the estimate on the performance of entropy-regularized natural policy gradients measured in the unregularized reward. 

\begin{proof}[Proof of \Cref{thm:overallError}]
Through a direct combination of \Cref{prop:errorDecomposition}, \Cref{thm:convergenceValue}, and \Cref{thm:convergenceRegularizedNPG}, we obtain~\eqref{eq:overallErrorBound}. 
\end{proof}

Recently, a sharp asymptotic analysis of entropy-regularized natural policy gradient methods has been conducted in~\cite{liu2024elementary} showing $R_\tau^\star - R(\pi_k)=O((1+\eta\tau)^{-2k})$ compared to the $O((1-\eta\tau)^{k})$ convergence of \Cref{thm:convergenceRegularizedNPG}. 
However, an analogous estimate on the convergence of the policies, as well as a control on the entry times after which the rate holds, is missing, which prevents us from using it in our analysis. 
By \Cref{thm:convregenceNPG} unregularized natural policy gradients achieve $O(k^{\gamma c}e^{-\Delta \eta k})$ convergence. 
Hence, although \Cref{thm:overallError} improves existing $O(\frac{\log k}{k})$ guarantees, the provided rate is still asymptotically slower compared to unregularized natural policy gradients. 
When we consider a fixed regularization and step size, $O(\gamma^k)$-convergence was established when exponentially decreasing the regularization and increasing the step size~\cite{li2023homotopic}, which can also be achieved by unregularized policy mirror descent, where the rate $O(\gamma^k)$ is known to be optimal~\cite{johnson2024optimal}.
For the small stepsize limit $\eta\to0$, the updates of the regularized natural policy gradient scheme follow the regularized Kakade gradient flow recently studied in~\cite{kerimkulov2023fisher}. 
These results complement the discrete-time analysis and ensure that $D_\K(\pi^\star_\tau, \pi_t) \le e^{-\tau t} D_\K(\pi^\star_\tau, \pi_0)$, see~\cite[Equation (61)]{kerimkulov2023fisher}. 
An analog treatment to the discrete-time case yields an overall estimate of $O(e^{-\sqrt{{2^{-1}\Delta t}}})$ for the regularized flow with strength $\tau = \sqrt{2\Delta t^{-1}}$, compared to the existing $O(t^{-1})$ guarantee in~\cite{sethi2024entropy}.

\section{Conclusions}

In this work, we provide an explicit characterization of the error introduced by entropy regularization in infinite-horizon discounted Markov decision processes with finite state and action spaces. 
We show that the solutions of the regularized optimization problems satisfy a gradient flow equation with respect to a Riemannian metric common in natural policy gradient methods. 
We use this to provide upper and lower bounds on the regularization error that decrease exponentially in the inverse regularization strength and match up to polynomial factors with an instance-dependent exponent. 
Thereby, we characterize the optimal exponential convergence rate for entropy regularization in discounted Markov decision processes.
Further, this correspondence allows us to identify the limit of this gradient flow as the generalized maximum entropy optimal policy, thereby characterizing the implicit bias of this gradient flow, which corresponds to a time-continuous version of the natural policy gradient method. 
We use our improved error estimates to show that for entropy-regularized natural policy gradient methods, the total error decays exponentially in the square root of the number of iterations.
    
\appendix

\section{Auxiliary Results}

\subsection{Performance difference lemma}
We recall an expression of the difference between the rewards of two policies in terms of the advantage function and the state-action distribution, and for a proof, we refer to~\cite[Lemma 2]{agarwal2021theory}.  

\begin{lemma}[Performance difference]\label{lem:PD}
    For any two policies $\pi_1, \pi_2\in\Delta_\bS^\A$ it holds that 
    \begin{align}
        R(\pi_1) - R(\pi_2) & = \frac{\langle \nu^{\pi_1}, A^{\pi_2} \rangle_{\bS\times\A}}{1-\gamma}. 
    \end{align}
\end{lemma}

\subsection{Pythagorean theorem in Bregman divergences}\label{sec:app:pythagoras}

For the sake of completeness, we provide the proof of a generalized Pythagorean theorem. 

\begin{proposition}[Pythagoras for Bregman divergences]
Consider a convex differentiable function $\phi\colon\Omega\to\R$ defined on a convex set $\Omega\subseteq\R^d$. 
Further, consider a convex and closed subset $X\subseteq\Omega$, fix $y\in\Omega$ as well as a Bregman projection
\begin{align}
    \hat{y} \in \argmin_{x\in X} D_\phi(x,y). 
\end{align}
Then for any $x\in X$ we have 
    \begin{align}
        D_\phi(x,y) \ge D_\phi(x,\hat y) + D_\phi(\hat y,y)
    \end{align}
    If further $X = \Omega \cap \mathcal L$ for some affine space $\mathcal L\subseteq\R^d$ and if $\hat{y}\in\operatorname{int}(X)$ we have 
    \begin{align}
        D_\phi(x,y) = D_\phi(x,\hat y) + D_\phi(\hat y,y) 
    \end{align}
\end{proposition}
\begin{proof}
Set $g(x)\coloneqq D_\phi(x, y)$, then by the first order stationarity condition for constrained convex optimization, it holds that 
\begin{align}\label{app:eq:first-order-stationarity}
    0\ge \nabla g(\hat y)^\top(\hat y - x) = (\nabla \phi(\hat y) - \nabla \phi(y))^\top(\hat y - x) \quad \text{for all } x\in X, 
\end{align}
where we used the definition of the Bregman divergence. 
We use this to estimate 
\begin{align}\label{app:eq:estimate}
    \begin{split}
        D_\phi(x,\hat y) + D_\phi(\hat y, y) & = \phi(x) - \phi(\hat y ) - \nabla \phi(\hat y)^\top(x-\hat y) + \phi(\hat y) - \phi(y) - \nabla \phi(y)^\top (\hat y - y) 
    \\ & 
    \le \phi(x) - \nabla \phi(y)^\top(x-\hat y) - \phi(y) - \nabla \phi(y)^\top (\hat y - y) 
    \\ & 
    = \phi(x) - \phi(y) - \nabla \phi(y)^\top(x-y) 
    \\ & 
    = D_\phi(x, y). 
    \end{split}
\end{align}
If $X = \Omega \cap \mathcal L$ for an affine $\mathcal L\subseteq\R^d$ and $\hat{y}\in\operatorname{int}(X)$, then equality holds in~\eqref{app:eq:first-order-stationarity} and thus in~\eqref{app:eq:estimate}. 
\end{proof}

\section{Sublinear Convergence of Unregularized Natural Policy Gradient}\label{app:sec:convergence}

Here, we generalize the sublinear convergence analysis for unregularized natural policy gradients of~\cite[Section 5.3]{agarwal2021theory} to non-uniform initial policies $\pi_0$ and unnormalized reward $r\in\R^{\bS\times\A}$. 

\begin{restatable}[Progress lemma]{lemma}{progress}
\label{lem:ascentNPG}
    For any initial distribution $\mu\in\Delta_\bS$ and the natural policy gradient updates~\eqref{eq:NPG_update} it holds that 
    \begin{align}
        R(\pi_{k+1}) - R(\pi_k) \ge  \frac{1-\gamma}{\eta} \sum_{s\in\bS}\mu(s) \log Z_k(s) \ge 0. 
    \end{align}
    In particular, for any $s\in\bS$ we have 
    \begin{align}
        V^{\pi_{k+1}}(s) - V^{\pi_{k}}(s) \ge  \frac{1-\gamma}{\eta} \log Z_k(s) \ge 0. 
    \end{align}
\end{restatable}
\begin{proof}
By Jensen's inequality, we have 
\begin{align*}
    \log Z_k(s) = \log \sum_{a\in\A} \pi_k(a|s) e^{\frac{\eta A^{\pi_{k}}(s,a)}{1-\gamma}} \ge \frac{\eta}{1-\gamma} \sum_{a\in\A} \pi_k(a|s) A^{\pi_{k}}(s,a) = 0
\end{align*}
for all $s\in\bS$ and $k\in\N$. 
Using the performance difference \Cref{lem:PD} together with $d_{\pi_{k+1}}(s)\ge (1-\gamma) \mu(s)$ we estimate 
\begin{align*}
    R(\pi_{k+1}) - R(\pi_{k}) & = \frac1{1-\gamma} \sum_{s\in\bS, a\in\A} d^{\pi_{k+1}}(s)\pi_{k+1}(a|s) A^{\pi_k}(s,a) 
    \\ & 
    = \eta^{-1}\sum_{s\in\bS, a\in\A} d^{\pi_{k+1}}(s)\pi_{k+1}(a|s) \log \frac{\pi_{k+1}(a|s) Z_k(s)}{\pi_k(a|s)} 
    \\ & 
    = \eta^{-1} D_K(\pi_{k+1}, \pi_k) + \eta^{-1}\sum_{s\in\bS} d^{\pi_{k+1}}(s)\log Z_k(s)
    \\ & 
    \ge \frac{1-\gamma}{\eta}\sum_{s\in\bS} \mu(s)\log Z_k(s). 
\end{align*}
\end{proof}

\sublinear*
\begin{proof}
First, note that the natural policy gradient iteration is independent of the initial distribution $\mu\in\Delta_\bS$. 
Using the performance difference \Cref{lem:PD} as well as the definition of the natural policy gradient step, we estimate the suboptimality for an arbitrary state $s_0\in\bS$ via 
\begin{align*}
   V^\star(s_0) - V^{\pi_k}(s_0) & = \frac{1}{1-\gamma} \sum_{s\in\bS, a\in\A} d^{\pi^\star}(s)\pi^\star(a|s) A^{\pi_k}(s,a) 
    \\ & 
    = \eta^{-1}\sum_{s\in\bS, a\in\A} d^{\pi^\star}(s)\pi^\star(a|s)\log \frac{\pi_{k+1}(a|s) Z_k(s)}{\pi_k(a|s)} 
    \\ & 
    = \frac{D_K(\pi^\star, \pi_k) - D_K(\pi^\star, \pi_{k+1})}{\eta} + \eta^{-1}\sum_{s\in\bS} d^{\pi^\star}(s)\log Z_k(s)
    \\ & 
    \le \frac{D_K(\pi^\star, \pi_k) - D_K(\pi^\star, \pi_k)}{\eta} + (1-\gamma)^{-1} \sum_{s\in\bS} d^{\pi^\star}(s)(V^{\pi_{k+1}}(s) - V^{\pi_{k}}(s)),
\end{align*}
where we used \Cref{lem:ascentNPG} with $\mu = d^{\pi^\star}$. 
By the progress \Cref{lem:ascentNPG} we know that $V^{\pi_k}(s)$ is non-decreasing in $k$ and hence we obtain 
\begin{align*}
    V^\star(s_0) - V^{\pi_k}(s_0) & \le \frac1{k} \sum_{l=0}^{k-1} (V^\star(s_0) - V^{\pi_l}(s_0)) 
    \\ & 
    \le \frac1{k} \left(\sum_{l=0}^{k} \frac{D_K(\pi^\star, \pi_{l+1}) - D_K(\pi^\star, \pi_l)}{\eta} + (1-\gamma)^{-1}\sum_{l=0}^{k} \sum_{s\in\bS} d^{\pi^\star}(s)(V^{\pi_{l+1}}(s) - V^{\pi_{l}}(s)) \right)
    \\ & 
    \le \frac1{k} \left(\frac{D_K(\pi^\star, \pi_{k})}{\eta} + (1-\gamma)^{-1} \sum_{s\in\bS} d^{\pi^\star}(s)(V^{\pi_{k}}(s) - V^{\pi_0}(s)) \right)
    \\ & 
    \le \frac1{k} \left(\frac{D_K(\pi^\star, \pi_{k})}{\eta} + 2 (1-\gamma)^{-1} \lVert r \rVert_\infty \right). 
\end{align*}
It holds that 
\begin{align*}
    A^\star(s,a) - A^{\pi}(s,a) 
    \le 
    Q^\star(s,a) - Q^{\pi}(s,a)
    \le 
    \gamma \lVert V^\star - V^\pi \rVert_\infty
\end{align*}
and similarly 
\begin{align*}
    A^\star(s,a) - A^{\pi}(s,a) & \ge V^\pi(s) - V^\star(s),
\end{align*}
which completes the proof. 
\end{proof}

\section*{Acknowledgments}
The authors acknowledge funding by the Deutsche Forschungsgemeinschaft (DFG, German Research Foundation) under the project number 442047500 through the Collaborative Research Center \emph{Sparsity and Singular Structures} (SFB 1481).
SC acknowledges funding by the Federal Ministry of Education and Research (BMBF) and the Ministry of Culture and Science of the German State of North Rhine-Westphalia (MKW) under the Excellence Strategy of the Federal Government and the Länder. 
JM acknowledges funding by the European Union (ERC, FluCo, grant agreement
No. 101088488). 
Views and opinions expressed are however those of the author(s) only and do not necessarily reflect those of the European Union or of the European Research Council. Neither the European Union nor the granting authority can be held responsible for them. 

\bibliographystyle{plain} 
\bibliography{references}

\end{document}